\documentclass[12pt,psfig]{article}
\usepackage[dvipsnames]{xcolor}
\usepackage[toc,page]{appendix}
\usepackage{amsmath, amsthm, amssymb,amsfonts,amscd}
\usepackage{graphicx,epsfig,latexsym}
\usepackage{cite,calc,xcolor,subfigmat,mathrsfs,dsfont}
\usepackage[pagebackref=false]{hyperref}
\usepackage{graphics}
\usepackage{subfigure}
\usepackage{graphicx,epsfig}
\usepackage{amsfonts}
\usepackage{amscd}
\usepackage{mathrsfs}
\usepackage{enumerate}
\usepackage{latexsym}
\usepackage{tikz}
\usepackage[all]{xy}
\newcommand{\Sl}{\mathfrak{sl}_2}
\newcommand{{\N}}{{\rm N}}
\newcommand{{\M}}{{\rm M}}
\newcommand{\D}{{\partial }}
\newcommand{\B}{\mathscr{B}}

\newcommand{\mb}[1]{\underline{#1}}

\newcommand{\NZ}{\mathbb{N}_0}

\newcommand{\Span}{{\rm span}}
\newcommand{\ad}{{\rm ad}}

\newcommand{\ie}{{\em i.e.,} }
\newcommand{\eg}{{\em e.g.,} }
\newtheorem{thm}{Theorem}[section]
\newtheorem{cor}[thm]{Corollary}
\newtheorem{lem}[thm]{Lemma}
\newtheorem{prop}[thm]{Proposition}
\theoremstyle{definition}
\newtheorem{defn}[thm]{Definition}
\newtheorem{exm}[thm]{Example}
\newtheorem{rem}[thm]{Remark}

\newtheorem{notation}[thm]{Notation}
\numberwithin{equation}{section}
\usepackage{color}
\newcounter{IssueCounter}
\newtheorem{Issue}[IssueCounter]{Issue}

\def\be {\begin{equation}}
\def\ee {\end{equation}}
\def\ba {\begin{eqnarray}}
\def\ea {\end{eqnarray}}
\def\bpr {\begin{proof}}
\def\epr {\end{proof}}
\def\bes {\begin{equation*}}
\def\ees {\end{equation*}}
\def\bas {\begin{eqnarray*}}
\def\eas {\end{eqnarray*}}

\begin{document}
\baselineskip=18pt
\renewcommand {\thefootnote}{\dag}
\renewcommand {\thefootnote}{\ddag}
\renewcommand {\thefootnote}{ }

\pagestyle{empty}
\begin{center}
                \leftline{}
                \vspace{-0.00 in}
{\Large \bf
Vector potential normal form classification for completely integrable solenoidal nilpotent singularities
} \\ [0.4in]

{\large Majid Gazor$^{*}$\footnote{$^*\,$Corresponding author. Phone: (98-31) 33913634; Fax: (98-31) 33912602; Email: mgazor@cc.iut.ac.ir.},
Fahimeh Mokhtari}

\vspace{0.15in}
{\small {\em Department of Mathematical Sciences,
Isfahan University of Technology
\\[-0.5ex]
Isfahan 84156-83111, Iran
}}
\vspace{0.30in}

{\large Jan A. Sanders}

\vspace{0.15in}
{\small {\em Department of Mathematics, Faculty of Sciences, Vrije Universiteit,
\\[-0.5ex]
Amsterdam 1081 HV, The Netherlands
}}

\today
\noindent
\end{center}

\vspace{-0.10in}

\baselineskip=15pt

\:\:\:\:\ \ \rule{5.88in}{0.012in}

\begin{abstract}

We introduce a \(\Sl\)-invariant family of nonlinear vector fields with a non-semisimple triple zero singularity. In this paper we are concerned with
characterization and normal form classification of these vector fields. We show that the family constitutes a Lie algebra structure and each vector field from
this family is solenoidal, completely integrable and rotational. All such vector fields share a common quadratic invariant. We provide a Poisson structure for
the Lie algebra from which the second invariant for each vector field can be readily derived. We show that each vector field from this family can be uniquely
characterized by two alternative representations; one uses a vector potential while
the other uses two functionally independent Clebsch potentials. Our normal form results are designed to preserve these structures and representations.
The results are implemented in Maple in order to compute vector potential and the Clebsch potential normal forms of a given vector field from this family.
Some practical normal form coefficient formulas for degrees of up to four are presented.
\end{abstract}

\vspace{-0.2in}
\vspace{0.10in} \noindent {\it Keywords}: Normal form classification; Triple zero singularity; \(\Sl\)-Lie algebra representation;
Clebsch potentials; Completely integrable vector field.

\vspace{0.10in} \noindent {\it 2010 Mathematics Subject Classification}:\, 34C20; 34A34.
\vspace{0.2in}

\section{Introduction}

We are concerned with nonlinear normal form classification of a \(\Sl\)-Lie algebra generated family of vector fields with a non-semisimple nilpotent
linear part, \ie \({\rm N}:={x}\frac{\D}{\D{{y}}}+2{y}\frac{\D}{\D {{z}}}\). Jacobson-Morozov theorem provides the other two generators, that
they form a triple along with \({\rm N},\) for a \(\Sl\)-Lie algebra. Consider \({\rm N}\) as a differential operator. Then by the adjoint action of
the \(\Sl\)-Lie algebra on (nonlinear) vector fields, we introduce a \(\Sl\)-invariant family of vector fields. We show that the set of
all such nonlinear vector fields whose linear part is a multiple scalar of \({\rm N}\) constitutes a Lie algebra, where we denote it by \(\B\).

An important goal in our normal form results is to detect, compute and preserve possible symmetries and geometrical features of a vector field's flow.
Hence, several geometric properties for our introduced \(\Sl\)-invariant family
are carefully studied in this paper. A natural dynamics analysis of such vector fields must
respect these geometrical features which have applications in different applied disciplines. However, the classical normal form computations
typically destroy certain symmetries of the truncated normal form system. These may include the system's properties such as
volume-preserving, Clebsch potentials, vector potentials, etc. Thus in either of these cases, the dynamics analysis of the truncated normal
form is not appropriate. Hence it is important to use permissible transformations which preserve the
system's symmetry; also see \cite{GazorMokhtariEul,GazorShoghi,GazorMokhtariInt}. One of our most important claimed contributions here is that our (truncated)
normal form results preserve all the different representations
(described below) in the paper such as vector potential, Clebsch potentials, and the volume-preserving property.

Solenoidal vector fields appear in disciplines such as magnetic fields and fluid mechanics; \eg see
\cite{kotiuga1991clebsch,longcope2005topological,mackay1994transport,haller1998reduction}.
Computing the invariants of such vector fields is an important goal in this paper. Any solenoidal vector field, say \(v,\) takes a
vector potential representation, that is, there exists a vector field, say \(w,\) whose curl is \(v,\) \ie \(\nabla\times w= v.\) We prove that all vector
fields in \(\B\) are solenoidal and provide the method and formulas for deriving their vector potential normal forms. We further introduce a Poisson algebra
that is Lie-isomorphic
to \(\B\) through a Lie isomorphism \(\psi.\) Another most important claimed contributions in this paper is that the Lie isomorphism \(\psi\) associates
a first integral \(\psi(v)\) to each vector field \(v\) from \(\B\). We further show that the quadratic polynomial \(\Delta:=xz-y^2\) is a second
first integral for all vector fields in \(\B.\)

Analytic normalization of an analytic vector field has close relations with complete integrability of the vector field;
\eg see \cite{Stolovitch,ZungConvergence,StolovitchSingular,ZhangAnalytic}. We recall that two first integrals for \(v\) in \(\B\) are called {\it functionally independent}
when their gradients have a rank of \(2\) for almost everywhere; \eg see \cite[page 3553]{ZhangAnalytic} and \cite{StolovitchSingular}. The level curves of these
invariants provide a comprehensive understanding about the orbits in the state space associated with the vector field's flow. Each vector field
\(v\) from the \(\Sl\)-invariant Lie algebra \(\B\) is a completely integrable solenoidal vector field; \ie we show that
the invariants \(\Delta\) and \(\psi(v)\) for each \(v\in \B\) are functionally independent. There is another alternative representation for completely integrable
solenoidal vector fields, that is given by the two of the vector field's
functionally independent invariants. Indeed, we prove that each vector field \(v\) in \(\B\) equals the exterior product of the gradients of \(\Delta= xz-y^2\)
and \(\psi(v)\); the latter is obtained through the Lie isomorphism \(\psi\) between \(\B\) and our introduced Poisson algebra.
The first integrals in the exterior product are referred as Clebsch potentials or Euler potentials of the vector field \(v\); \eg see
\cite{longcope2005topological,kotiuga1991clebsch}. We refer to \(\Delta\) by the primary
Clebsch potential and \(\psi(v)\) as the secondary Clebsch potential for \(v.\) We further conclude that these families of triple zero singularities
are rotational vector field, that is, their curl is non-zero. This implies that these are not gradient vector fields.

Finally, we prove that \(\B\) is the set of all multiple scalars
of solenoidal vector fields such as \(v,\) that is given by
\be\label{eq1}
{v}:=-{x}\frac{\D}{\D{{y}}}-2{y}\frac{\D}{\D {{z}}}+{\rm v}(x, y, z),
\ee where \({\rm v}: \mathbb{R}^3\rightarrow \mathbb{R}^3\) denotes a vector field without constant and linear parts,
\be\label{eq2}{\rm div}(v)=0, \qquad\hbox{and }\qquad v(\Delta)=0\quad\hbox{ where }\quad\Delta= xz-y^2.\ee
Note that for our convenience we interchange the uses of notations and terminologies such as vector fields, differential systems and differential operators.
We also remark that \(x\) and \(\Delta\) are the
two generators of the invariant algebra for the linear vector field \({\rm N}.\) We refer to the vector field \(v\) in equation \eqref{eq1}-\eqref{eq2} by a completely
integrable system since it has two functionally independent invariants \(\Delta\) and \(\psi(v)\). The Poisson structure and the Lie isomorphism \(\psi\)
provide a practical method for deriving the second first integral within the invariant algebra of vector fields given by \eqref{eq1}-\eqref{eq2} and their normal forms.

\markright{{\footnotesize {\it M. Gazor, F. Mokhtari and J.A. Sanders \hspace{1.5in} {\it Completely integrable nilpotent singularities}}}}
\pagestyle{myheadings}

Normal form classification of nilpotent singularities has been a challenging task. Even in the two--dimensional case, there have been
numerous important contributions in various types and approaches; \eg see
\cite{baidersanders,BaidSand91,dumortier2001new,GazorSadri17,GazorMoazeni,Zoladek03,Zoladek02,ZoladekNF2015,Stroyzyna,yy2001b,yy2001,kw,wlhj,ChenWang,chendora,Algaba}.
There have only been a few contributions in three dimensional state space cases; see \cite{YuYaun3Zero,Algaba30} where hypernormalization
is performed up to degree three; also see \cite{GazorMokhtariEul,GazorMokhtari,GazorSadri} and \cite{MurdBook,MurdockFinal,murdock2016box,Murd04,Murd08}.
In this paper we provide a complete normal form classification for all vector fields \(v\) in equations
\eqref{eq1}-\eqref{eq2}, that is, the set of all completely integrable solenoidal nilpotent singularities where \(\Delta\) is one of their
invariants and a multiple scalar of \({\rm N}\) is their linear part. These vector fields and their normal forms are uniquely characterized by
their secondary Clebsch potential. Indeed, the primary Clebsch potential \(\Delta\) is always preserved throughout the normalization steps
while the normalizing transformations naturally reflect the normal form changes into the secondary Clebsch potential. In Theorem \ref{UNF}, we prove that a
vector field given by \eqref{eq1}-\eqref{eq2} can be either linearized or uniquely transformed into the formal normal form vector field
\ba\label{SNFInt}
&{\bf w}:=-{x}\frac{\D}{\D{{y}}}-2{y}\frac{\D}{\D{{z}}}+{{z}}^{p} ({z}\frac{\D}{\D{{y}}}
+2{y}\frac{\D}{\D{{x}}})+\sum_{k=p+1}^{\infty}\sum_{i=0}^{[{\frac{k+p}{2}}]} b_{i,k}{{z}}^i(xz-y^2)^{k-2i+p}({z}\frac{\D}{\D{{y}}}
+2{y}\frac{\D}{\D{{x}}}),&
\ea where \(b_{i,k}\in \mathbb{R}\) and \(p\) is a natural number. In addition, the secondary Clebsch potential normal form is given by
\ba\label{invSNFInt}
&I(x, y, z)={x}+\frac{1}{p+1}{{z}}^{{p}+1}+\sum_{k=p+1}^{\infty}\sum_{i=0}^{[{\frac{k+p}{2}}]}
{\frac{b_{{{i},k}}}{{i}+1}}{{z}}^{{i}+1}{( {x}{z}-{{y}}^{2} )}^{k-2i+p}.&
\ea The normal form invariant \eqref{invSNFInt} can sometimes be used for further reduction of the normal form vector fields; \eg see section \ref{Sec5}.

Now we describe the organization of the rest of this paper. We introduce a family of \(\Sl\)-invariant irreducible vector spaces of vector
fields in section \ref{Sec2}. We further prove that this family constitutes a Lie algebra and derive their associated structure constants.
Section \ref{Sec3} is devoted for introducing a Poisson algebra and prove that it is Lie isomorphic to \(\B.\)
Next, we discuss the geometrical properties of \(\B\) family in section \ref{Sec4}. In particular, we show that our \(\Sl\)-invariant introduced family of vector fields
are fully characterized by equations \eqref{eq1}-\eqref{eq2}. Two further representations for each such vector field are presented in this section
by using their Clebsch potentials and vector potentials. Section \ref{Sec5} is dedicated to study the normal form classification for vector fields \eqref{eq1}-\eqref{eq2}.
Some practical formulas for normal form coefficients of up to degree three for a given triple zero singularity \eqref{eq1}-\eqref{eq2} is presented. Finally, we introduce
two more \(\Sl\)-invariant families of vector fields in Appendix \ref{SecAppend}. These along with the family given by equations \eqref{eq1}-\eqref{eq2}
would amount to three family types of \(\Sl\)-invariant vector fields so that each three dimensional vector field can be uniquely Taylor expanded in terms of vector
polynomial generators from these three types.

\section{Algebraic structures }\label{Sec2}

Let
\ba\label{Triad}
&\N :={x}\frac{\D}{\D {y}}+2 {y} \frac{\D}{\D {z}}, \qquad \M:= {z} \frac{\D}{\D {y}}+ 2 {y} \frac{\D}{\D {x}}, \quad \hbox{and }
\quad {\rm H}:= [{\rm M}, {\rm N}]= {\rm M} {\rm N}-{\rm N}{\rm M}=2 {z} \frac{\D}{\D {z}} - 2{x} \frac{\D}{\D {x}}.&
\ea
The triple \(\{\M,\N, {\rm H}\}\) generates a \(\Sl\) Lie algebra, \ie
\bes[\M,\N]={\rm H}, \qquad [{\rm H},\M]=2\M, \qquad [{\rm H},\N]=-2\N.\ees
We denote \(\N^n\mb f\, v=(\N^n f)\, v\) for the iterative action of \({\rm N}\) as a differential operator on the scalar function \(f\) that is also multiplied with \(v.\)
Further for a vector field \(v,\) \(\N^n  \mb{ v}\) is inductively defined by
\bes
{\rm N} \mb{ v}:={\rm ad}_{\N}  { v}, \quad\hbox{ and }\quad \N^n  \mb{ v}:= {\rm ad}_{\N} {\N}^{n-1}  \mb {v} \qquad \hbox{ for } n>1.
\ees We often put the underline only on the last element on which the operator acts, \ie \(\N^n f\mb gh:= h\N^n \mb{fg}\) and
\({\rm N}f \mb{\rm M}:= {\rm N} \mb{f\rm M}\). Note that \({\rm N}\) as an operator distinguishes vector fields from
scalar functions: the operator \(N\) merely acts on scalar functions as a differential operator while it acts as a Lie operator on vector fields.
By \cite[Proposition 2]{CushSandCont}, \(\mathbb{R}[[{z}, \Delta]]\) is the invariant ring for \({\rm M}\).
For a homogeneous scalar polynomial function \(f:\mathbb{R}^3[{x},{y},{z}]\rightarrow\mathbb{R}\) in \(\ker {\rm M}=\left\langle{\Delta,{z}}\right\rangle,\)
there exists an integer \( \omega_f\in \mathbb{Z}\) so that
\bas
{\rm H} f= \omega_f f,
\eas where \(\omega_f\)  and \(f\) are called  the  eigenvalue and eigenfunction of the differential operator \({\rm H},\) respectively.
The algebra of first integrals for \(\M\) is the same as \(\ker {\rm M}= \left\langle\Delta, {z} \right\rangle\); see \cite{CushmanSanders},
\cite[Chapter 2]{MurdBook} and  \cite[Chapter 9]{SandBook} for more details on the representation of \(\Sl\) and normal form theory. In this section
we use the \(\Sl\)-triple \eqref{Triad} to generate a family of irreducible \(\Sl\)-invariant vector spaces. The set of all such invariant vector
spaces constitutes a Lie algebra. This consists of all completely integrable  and solenoidal vector fields of triple zero singularities, that is,
their linear part is a scalar multiple of \(N\) in \eqref{Triad}. Vector fields from this family can be considered as alternatives in three
dimensional state space for completely integrable Hamiltonian systems which always require even dimensions in state space; also see
\cite[page 2812]{GazorMokhtariInt}.

\begin{notation}\label{not1}
\begin{itemize}
\item The following notations frequently appear in this paper.
\ba\label{sigma12}
&\sigma_1:=\sigma_1(s_1,s_2,k_1,k_2)=s_1+s_2+k_1+k_2, \quad
\sigma_2:=\sigma_2(q_1,q_2,i_1,i_2)=q_1+q_2-i_1-i_2.&
\ea
\item We denote \({{\bf e}_1}, {{\bf e}_2}\) and \({{\bf e}_3}\) for the standard basis of \(\mathbb{R}^3\) and \(\kappa_{l,i}:=\frac{i !}{(i-l)!}.\)
\item We use the Pochhammer \(k\)-symbol notation for any \(a, b\in\mathbb{R}, k\in \mathbb{N}\) as \(\left(a\right)^k_b:=\prod_{j=0}^{k-1}(a+jb).\)
\item  Throughout this paper we frequently use some constants or variables with negative powers in the denominator (or numerator) of a fraction. The reader should
merely treat this as a formal misuse of notation to shorten the formulas.
\end{itemize}
\end{notation}
Now we present some technical results which play a central role in this paper.
\begin{lem}\label{NM}
Let \(f\) be a homogeneous scalar polynomial function. Then,
\begin{eqnarray}
\N^nf\mb{\M}&=& \N^{n}\mb{f} {\M}- n \N^{n-1}\mb{f} {\rm H}-n(n-1)\N^{n-2}\mb{f} {\N},\quad \hbox{for any } \;\; n\in \mathbb{N}_0.
\end{eqnarray}
 \end{lem}

\begin{proof}
The proof is by induction. For  \(n=1,\) we have
\bas
\N f\mb{\M}=\N(f){\M}+f[\N,{\M}]= \N(f){\M}-f{\rm H}=\N\mb f{\M}-f{\rm H}.
\eas
By the  induction hypothesis  we have
\bas
{\rm N}^{n+1}f\mb{\M}&=&[{\rm N}, {\rm N}^{n}\mb{f} {\M}- n {\rm N}^{n-1}\mb{f} {\rm H}- n(n-1){\rm N}^{n-2}\mb{f} {{\rm N}}]
\\
&=&
{\rm N}^{n+1}\mb{f} {\M}+ {\rm N}^{n}\mb{f} {{\rm N}\M}- n {\rm N}^{n}\mb{f} {\rm H}- n {\rm N}^{n-1}\mb{f} {{\rm N}\rm H}- n(n-1){\rm N}^{n-1}\mb{f} {{\rm N}}
\\
&&- n(n-1){\rm N}^{n-2}\mb{f} {{\rm N}{\rm N}}-({\rm N}^{n}\mb{f} {\M{\rm N}}- n {\rm N}^{n-1}\mb{f} {\rm H\N}- n(n-1)\N^{n-2}\mb{f} {{\rm N}{\rm N}})
\\
&=&
{\rm N}^{n+1}\mb{f} {\M}+{\rm N}^{n}\mb{f} {[{\rm N},\M]}- n {\rm N}^{n}\mb{f} {\rm H}- n {\rm N}^{n-1}\mb{f} {[\N,\rm H]}- n(n-1)\N^{n-1}\mb{f} {\N}
\\
&=&
{\rm N}^{n+1}\mb{f} {\M}- (n+1) {\rm N}^{n}\mb{f} {\rm H}- n(n+1){\rm N}^{n-1}\mb{f} {{\rm N}}.
\eas
This proves the statement; also see \cite[Proposition 1]{CushmanSanders}.
\end{proof}
There is an  important corollary to this lemma.
\begin{cor}\label{NnfM}
For any homogeneous function \(f\in \ker {\rm M},\) we have
\begin{eqnarray}\label{NfM}
&\N^nf\mb{\M}=\frac{2(\omega_f-n+2)\kappa_{n, \omega_f+2}\N^{n+1}{z}\mb{f}}{(\omega_f+2)\kappa_{n+1,\omega_f+2}} \frac{\D}{\D {x}}
+\frac{(\omega_f-2n+2)\kappa_{n,\omega_f+2}\N^{n}{z}\mb{f}}{(\omega_f+2)\kappa_{n, \omega_f+2}}\frac{\D}{\D {y}}
- \frac{2n\kappa_{n,\omega_f+2}\N^{n-1}{z} \mb{f}}{(\omega_f+2)\kappa_{n-1,\omega_f+2}}\frac{\D}{\D {z}}.&
\end{eqnarray}
\end{cor}
\begin{lem}\label{A.4}
For each \(l\in\mathbb{N}_0,\) the following equalities hold:
\bas
&\frac{\D}{\D x }\N^l= l(l-1) \N^{l-2}\frac{\D}{\D z}+ l\N^{l-1}\frac{\D}{\D y} + \N^l\frac{\D}{\D x},\qquad
\frac{\D}{\D y} \N^{l}= 2 l \N^{l-1}\frac{\D}{\D z} +\N^{l} \frac{\D}{\D y},\qquad
\frac{\D}{\D z} \N^l= \N^l \frac{\D}{\D z}.&
\eas
\end{lem}

\begin{proof} The proof is by induction on \(l.\) For instance, by the induction hypothesis we have
\begin{eqnarray*}
&\frac{\D}{\D y}\N^{l+1}= 2l\N^{l-1}\frac{\D}{\D z}\N+ \N^{l} \frac{\D}{\D y} \N=2l\N^{l}\frac{\D}{\D z}+2\N^{l}\frac{\D}{\D z} + \N^{l}\N\frac{\D}{\D y}
=2(l+1)\N^{l}\frac{\D}{\D z}+\N^{l+1}\frac{\D}{\D y}.&
\end{eqnarray*}
This proves the second equality.
\end{proof}

\begin{lem}\label{A.1}
Let \(q=2s+r,\) where \(r=0\) or \(r=1\) and \(s \in\mathbb{N}_0.\) Then,
\ba\label{Nq1}
\N^q\mb{{z}^i}&=&\sum_{n=0}^s\eta^{q,i}_n{x}^{s-n}{y}^{r}{z}^{i-s-r-n}\Delta^n, \quad \N^q\mb{{z}^i}=\sum_{n=0}^{s}\zeta^{q,i}_n {x}^{s-n}{y}^{2n+r}{z}^{i-n-s-r},
\ea where
\ba\label{Nq11}
{\eta^{q,i}_n}&:=&\frac{(-1)^{n}(s)_{-1}^n(i)_{-1}^{s+n+r}(2i-1)_{-2}^{s}{2}^{s+n+r}}{n! (2i-1)_{-2}^{n}} \quad \hbox{ and }
\quad\zeta^{q,i}_n:={\frac{i!(2s+r)! {2}^{2n+r}}{(s-n)!(2n+r)!(i-n-s-r)!}}\quad
\ea for \(i\geq s+n+r,\) while \({\eta^{q,i}_n}:= \zeta^{q,i}_n:=0\) for \(i<q+n-s.\) In particular, \(\N^q\mb{{z}^i}=0\) when \(s<q-i.\)
\end{lem}
\bpr
The proof is straightforward by an induction on \(q.\)
\epr Lemma \ref{A.1} enables us to formulate \(\N^{q_1}\mb{{z}^i}\N^{q_2}\mb{{z}^j}\) in terms of \(x, y, z\) and \(\Delta.\)
\begin{prop}\label{dnqnd}
Let \(q_1=2s_1+r_1\) and \(q_2=2s_2+r_2,\) where \(r_1, r_2 \in\{0, 1\}\) and  \(s_1, s_2\in \mathbb{N}_0\). Then,
\ba\label{Nqq1Nqq1}
&\N^{q_1}\mb{{z}^i}\N^{q_2}\mb{{z}^j}=\sum_{n=0}^{\min\{s_1+s_2-\sigma_2, s_1+s_2\}}
\tilde{\eta}^{q_1,q_2}_{n,i,j}\frac{{x}^{s_{{1}}+s_{{2}}}{y}^{r_{{1}}+r_{{2}}}{z}^{i+j}}{{x^nz^{s_1+s_2+r_1+r_2+n}}}{\Delta}^{n},&
\ea and
\bas
&\N^{q_1}\mb{{z}^i}\N^{q_2}\mb{{z}^j}= \sum_{n=0}^{\min\{s_1+s_2-\sigma_2, s_1+s_2\}}\tilde{\zeta}^{q_1,q_2}_{n,i,j}
\frac{x^{s_1+s_2}y^{2n+r_1+r_2}z^{i+j}}{x^nz^{s_1+s_2+r_1+r_2+n}}&
\eas hold where
\ba\label{z1}
\tilde{\zeta}^{q_1,q_2}_{n,i,j}&:=&\sum_{r=0}^n\zeta^{q_1,i}_r\zeta^{q_2,j}_{n-r}, \qquad \tilde{\eta}^{q_1,q_2}_{n,i,j}:={\sum_{r=0}^{n}\eta^{q_1,i}_r\eta^{q_2,j}_{n-r}}.
\ea Further, \(\N^{q_1}\mb{{z}^i}\N^{q_2}\mb{{z}^j}=0\) when either \(s_1<q_1-i\) or \(s_2<q_2-j.\)
In particular, \(\N^{q_1}\mb{{z}^i}\N^{q_2}\mb{{z}^j}=0\) for \(s_1+s_2< \sigma_2(q_1, q_2, i , j)=q_1+q_2-i-j.\)
\end{prop}
\bpr
Given Lemma \ref{A.1}, the proof is trivial.
\epr

The following theorem provides an alternative formula for the expansion of \(\N^{q_1}\mb{{z}^i}\N^{q_2}\mb{{z}^j}\).

\begin{thm}\label{A.3}
Let \(q_1=2s_1+r_1,\) \(q_2=2s_2+r_2\) and \(\sigma_2(q_1, q_2, i , j)=q_1+q_2-i-j.\) Then,
\ba\label{nn}
\N^{q_1}\mb{{z}^i}\N^{q_2}\mb{{z}^j}&=&
\sum_{p=\max \{\sigma_2, 0\}}^{s_1+s_2+\lfloor\frac{r_1+r_2}{2}\rfloor}C_{p,i,j}^{q_1,q_2}\N^{2p+|r_2-r_1|}
\mb{{z}^{2p-\sigma_2+|r_2-r_1|}}\Delta^{s_1+s_2-p+{\lfloor\frac{r_1+r_2}{2}\rfloor}},
\ea where \(C^{q_1,q_2}_{p,i,j}\) is given by
\ba\label{Cq12pij}
&\sum_{r=0}^{s_1+s_2+\lfloor\frac{r_1+r_2}{2}\rfloor-p}\frac{\big(\tilde{\eta}^{q_1,q_2}_{r,i,j}-\lfloor\frac{r_1+r_2}{2}\rfloor\tilde{\eta}^{q_1,q_2}_{r-1,i,j}\big)
{(p+1)^{{s_1+s_2-p-r+\lfloor\frac{r_1+r_2}{2}\rfloor}}_1(p-\sigma_2+1-|r_2+r_1|)_1^{{s_1+s_2-p-r+\lfloor\frac{r_1+r_2}{2}\rfloor}} }}{\eta^{2p+|r_2-r_1|,2p-\sigma_2}_0
(s_1+s_2-p-r)!{2^{p+r-(s_1+s_2+\lfloor\frac{r_1+r_2}{2}\rfloor)}}        		(4p-2\sigma_2+3-4\lfloor\frac{r_1+r_2}{2}\rfloor)_2^{s_1+s_2-p-r
+\lfloor\frac{r_1+r_2}{2}\rfloor}}.&
\ea
\end{thm}

\bpr A polynomial expansion for the left hand side in \eqref{nn} is derived in equation \eqref{Nqq1Nqq1} while by using the first equation in \eqref{Nq1}, the
right hand side is given by
\bas
&\sum_{p=\max\{\sigma_2,0\}}^{s_1+s_2+\lfloor\frac{r_1+r_2}{2}\rfloor}C_{p,i,j}^{q_1,q_2}
\sum_{n=0}^{p}\eta^{2p+|r_2-r_1|,2p-\sigma_2+|r_2-r_1|}_{n}{x}^{p-n}{y}^{|r_{{2}}-r_{{1}}|}{z^{p-\sigma_2-n}}{\Delta}^{s_1+s_2-p+n+\lfloor\frac{r_1+r_2}{2}\rfloor}.&
\eas Given Lemma \ref{A.1}, we remark that \(\N^{2p+|r_2-r_1|}\mb{{z}^{2p-\sigma_2+|r_2-r_1|}}=0\) for \(p<\sigma_2.\)

When \(r_1r_2=0,\) equation \eqref{nn} is equivalent with the following polynomial equation
\bas
&\sum_{n=0}^{\min\{s_1+s_2-\sigma_2, s_1+s_2\}}\tilde{\eta}^{q_1,q_2}_{n,i,j}{x}^{s_{{1}}+s_{{2}}-n}{z}^{i+j-(s_{{1}}+s_{{2}}+r_{{1}}+r_{{2}}+n)}\Delta^{n}&
\\&=\sum_{n=0}^{\min\{s_1+s_2-\sigma_2, s_1+s_2\}}\sum_{p=s_1+s_2-n}^{s_1+s_2}
C_{p,i,j}^{q_1,q_2}\eta^{2p+|r_2-r_1|,2p+i+j-2(s_1+s_2)}_{n+p-(s_1+s_2)}{x}^{s_1+s_2-n}{z^{i+j-(s_1+s_2+r_1+r_2+n)}}{\Delta}^{n}.&
\eas Hence for each \(i, j, s_1, s_2,\) and \(0\leq n\leq \min\{s_1+s_2-\sigma_2, s_1+s_2\},\) we have
\bes
\tilde{\eta}^{q_1,q_2}_{n,i,j}
=\sum_{k=0}^{n}\eta^{q_1+q_2-2k,i+j-2k}_{n-k}C_{s_1+s_2-k,i,j}^{q_1,q_2}.\ees
These introduce a family of upper triangular linear matrix equations. The determinant of the coefficient matrix is given by
\(\prod_{n=0}^{\min\{s_1+s_2-\sigma_2, s_1+s_2\}}\eta^{q_1+q_2-2n,i+j-2n}_{0}\neq0.\) These together with the first equation in \eqref{Nq11} give rise to
\bas
&C_{p,i,j}^{q_1,q_2}=\sum_{r=0}^{s_1+s_2-p}\frac{\tilde{\eta}^{q_1,q_2}_{r,i,j}{(p+1)^{{s_1+s_2-p-r+}}_1(p-\sigma_2+1-|r_2+r_1|)_1^{{s_1+s_2-p-r}}}}
{\eta^{2p+|r_2-r_1|,2p-\sigma_2}_0(s_1+s_2-p-r)!{2^{p+r-(s_1+s_2)}}(4p-2\sigma_2+3)_2^{s_1+s_2-p-r}}.&
\eas
Now let \(r_1=r_2=1.\) By substituting \(y^2=xz-\Delta\) into equation \eqref{Nqq1Nqq1}, \(\N^{q_1}\mb{{z}^i}\N^{q_2}\mb{{z}^j}\) is given by
\bas
&\sum_{n=1}^{\min\{s_1+s_2-\sigma_2,s_1+s_2\}}(\tilde{\eta}^{q_1,q_2}_{n,i,j}-\tilde{\eta}^{q_1,q_2}_{n-1,i,j}){x}^{s_{{1}}
+s_{{2}}-n+1}{z}^{i+j-s_{{1}}-s_{{2}}-n-1}{\Delta}^{n}+\tilde{\eta}^{q_1,q_2}_{0,i,j}{x}^{s_{{1}}+s_{{2}}+1}{z}^{i+j-s_{1}-s_{2}-1}&\\&
-\tilde{\eta}^{q_1,q_2}_{\min\{s_1+s_2-\sigma_2,s_1+s_2\},i,j}x^{\max\{0, \sigma_2\}}{z}^{\max\{-\sigma_2,0\}}{\Delta}^{s_1+s_2+1}.
\qquad\qquad\qquad\qquad\qquad\qquad\qquad\qquad\quad&
\eas Hence the family of linear equations and its solutions are derived by
\bas
&\tilde{\eta}^{q_1,q_2}_{n,i,j}-\tilde{\eta}^{q_1,q_2}_{n-1,i,j}=\sum_{k=0}^n  \eta^{q_1+q_2-2k,i+j-2k}_{n-k}C_{s_1+s_2-k,i,j}^{q_1,q_2}, \qquad \hbox{ and } \qquad &
\\&C^{q_1,q_2}_{n,i,j}:=\sum_{r=0}^{s_1+s_2+1-n}\frac{\big(\tilde{\eta}^{q_1,q_2}_{r,i,j}-\tilde{\eta}^{q_1,q_2}_{r-1,i,j}\big)
{(n+1)^{{s_1+s_2-n-r+1}}_1(n-\sigma_2-1)_1^{{s_1+s_2-n-r+1}} }}{(s_1+s_2-n-r)!{2^{n+r-(s_1+s_2+1)}\eta^{2n,2n-\sigma_2}_0}(4n-2\sigma_2-1)_2^{s_1+s_2-n-r+1}},&
\eas respectively.
\epr

\begin{lem} The formal power series ring \(\mathbb{R}[[z, \Delta]]\) is the ring of invariants for \(\M.\)
\end{lem}
\begin{proof} Obviously, \(\mathbb{R}[[z, \Delta]]\) is an invariant ring for \(\M.\) We prove that this is actually the ring of invariants for \(\M\) using
generating functions; see \cite[Lemma 4.7.9]{MurdBook} and \cite{CushSandCont,CushSandsur,CushmanSanders,CushSandWh,SandBook}. The generating function for \(\mathbb{R}[[z, \Delta]]\) is given by
\be\label{GenFunc}
\frac{1}{(1-u^2t)(1-t^2)}.
\ee Here, \(z\) has an eigenvalue \(2\) while the eigenvalue for \(\Delta\) is \(0.\) Hence, differentiating with respect to
\(u\) at \(u=1\) leads to
\bas
&\frac{(1-u^2t)(1-t^2)+2u^2t(1-t^2)}{(1-u^2t)^2(1-t^2)^2}|_{u=1}=\frac{(1-t)(1-t^2)+2t(1-t^2)}{(1-t)^2(1-t^2)^2}= \frac{1}{(1-t)^3}.&
\eas The latter is the generating function for all formal power series associated with three variables. The proof is complete by a differential form version of
\cite[Lemma 4.7.9]{MurdBook}.
\end{proof}

\begin{thm}\label{3DVects} Let \(V= \Span \{{\rm N}^nz^i \Delta^k \mb{\frac{\D}{\D x}}, {\rm N}^nz^i \Delta^k \mb \M, {\rm N}^nz^i \Delta^k\mb {\rm E}\,| n, i, k\in \NZ\}\)
and
\bas
&\mathcal{K}=\Span \{z^i \Delta^k \frac{\D}{\D x}, z^i \Delta^k \M, z^i \Delta^k {\rm E}\,| i, k\in \NZ\},&
\eas where \(\NZ\) denotes nonnegative integers. Then, \(\mathcal{K}= \ker {\rm \ad}_\M\) and \(V\) is the set of all three dimensional formal vector fields.
\end{thm}
\bpr We first claim that the homogeneous polynomial set
\bas
&\{{\rm N}^nz^{i+1} \Delta^k \mb{\frac{\D}{\D x}}, {\rm N}^nz^i \Delta^k \mb \M, {\rm N}^nz^i \Delta^k \mb {\rm E}\,| n, i, k\in \NZ\}&
\eas is a linearly independent set. Since \([{\rm N}, f X]= {\rm N}\mb fX +f [{\rm N}, X]\) and \([{\rm N}, \M]=[{\rm N}, {\rm E}]=0,\) we imply that the set of polynomials
\({\rm N}^nz^i \Delta^k \mb \M\) and \({\rm N}^nz^i \Delta^k \mb {\rm E}\) for different \(n, i, k\) are linearly independent. On the other hand due to
\({\ad_{\rm N}}^3 \frac{\D}{\D x}= - {\ad_{\rm N}}^2\frac{\D}{\D y}= -2\,\ad_{\rm N}\frac{\D}{\D z}=0,\)
\bas
&{\rm N}^n z^{i+1} \Delta^k \mb{\frac{\D}{\D x}}= {\rm N}^n\mb z^{i+1}\Delta^k \frac{\D}{\D x}
-{\rm N}^{n-1}\mb z^{i+1}\Delta^k \frac{\D}{\D y}-2{\rm N}^{n-2}\mb z^{i+1}\Delta^k \frac{\D}{\D z}.&
\eas Then, the proof of our claim is complete by the linear independency of polynomials \({\rm N}^n\mb z^{i+1},\) \(2y{\rm N}^n\mb z^i\) and \(x{\rm N}^n\mb z^i.\)

Since \(\M\) commutes with \(\M,\) \({\rm E},\) \(\frac{\D}{\D x},\) and \(\M\mb z^i=0,\) \(\mathcal{K}\) is a subspace for \(\ker {\rm ad}_\M.\) By equation \eqref{GenFunc},
the generating function for \(\mathbb{R}[[z, \Delta]]\) is \(\frac{1}{(1-u^2t)(1-t^2)}\) and thus, the generating function for \(\ker {\rm ad}_\M\) is given by
\(\frac{3}{(1-u^2t)(1-t^2)}.\) This concludes that \(\mathcal{K}=\ker {\rm ad}_\M.\) Furthermore, \(V\) is the space of all three dimensional
vector fields with three variables and its generating function is given by \(\frac{1}{(1-t)^3}.\)
\epr

We define
\ba\label{bm}
{\rm B}^{l}_{i,k}:= \frac{ \N^{l+1}{{z}}^{i}\Delta^k\mb{\M}}{{\kappa_{l+1, 2i+2}}}, \qquad \hbox{for }\quad -1\leq l\leq 2i+1, i, k\in \mathbb{N}_0=\mathbb{N}\cup\{0\}.
\ea By taking \(f:={{z}}^{i}\Delta^k\) in equation \eqref{NfM}, \(\omega_f=2i\) and
\ba\label{Blik}
&{\rm B}^{l}_{i,k}=\frac{(2i-l+1) {\rm N}^{l+2}\mb{{{{z}}^{i+1}\Delta^k}}}{(i+1)\kappa_{l+2,2i+2}}\frac{\D}{\D {x}}
+\frac{(i-l){\rm N}^{l+1} \mb{{{{z}}^{i+1}\Delta^k}}}{(i+1)\kappa_{l+1,2i+2}}\frac{\D}{\D {y}}
-\frac{(l+1){\rm N}^{l}\mb{{{z}}^{i+1}\Delta^k}}{(i+1)\kappa_{l,2i+2}}\frac{\D}{\D {z}}.&
\ea Hence, \({\rm B}^1_{0, 0}:= -{\rm N}\) and \({\rm B}^{-1}_{0, 0}:= -{\rm M}.\) Now we introduce \(\B\) as the vector space spanned by all
nonlinear vector fields from this family with the nilpotent linear part \({\rm B}^1_{0,0}=\N\), \ie
\begin{eqnarray}\label{Bv}
\mathscr{B}:= \Span \left\{ {\rm B}^1_{0,0}+
\sum b^l_{i,k}{{\rm B}}^{l}_{i,k}\big|  -1\leqslant l \leqslant 2i+1, i\in \mathbb{N}, k\in\mathbb{N}_0, b^l_{i,k}\in{\mathbb{R}}\right\}.
\end{eqnarray}
Note that two more families of vector fields associated with Theorem \ref{3DVects} are defined in Appendix \ref{SecAppend}.

\begin{thm} Denote \(\sigma_1(s_1,s_2,k_1,k_2)\) and \(\sigma_2(q_1,q_2,i_1,i_2)\) for the indices defined in equations \eqref{sigma12}.
The vector space \(\B\) is a Lie algebra with structure constants given by
\ba\label{LieEq}
[{\rm B}^{q_1}_{i_1,k_1},{\rm B}^{q_2}_{i_2,k_2}]=
\sum_{j=\max \{\sigma_2-1, -1\}}^{s_1+s_2+{\lfloor\frac{r_1+r_2}{2}\rfloor}}{{a}}_{j,i_1,i_2}^{q_1,q_2}{\rm B}^{2j+|r_2-r_1|}_{2j-\sigma_2+|r_2-r_1|,\sigma_1-j
+{\lfloor\frac{r_1+r_2}{2}\rfloor}},
\ea where
\bas
&{{a}}_{j,i_1,i_2}^{q_1,q_2}=\frac{\sum^3_{p=1}\big({l}_{p,3,i_2,i_1}^{q_2,q_1}C_{j,i_2+1,i_1}^{q_2+3-p,{q_1}-3+p}
-{l}_{p,3,i_1,i_2}^{q_1,q_2}C_{j,i_1+1,i_2}^{q_1+3-p,{q_2}-3+p}\big)}
{{(2j+|r_2-r_1|+1)}(2j+1-\sigma_2+|r_2-r_1|)^{-1}{\kappa_{2j+|r_2-r_1|,4j-2\sigma_2+2+2|r_2-r_1|}}^{-1}},&
\eas for \(2j+1+|r_2-r_1|\neq0,\) while for \(j=-1,\) \(q_1=2s_1+1,\) and \(q_2=2s_2,\)
\ba\label{a-1}
&{{a}}_{-1,i_1,i_2}^{q_1,q_2}=|r_1-r_2|\frac{\sum^3_{p=1}\big({l}_{p,2,i_2,i_1}^{q_2,q_1}C_{0,i_2+1,i_1}^{q_2+3-p,{q_1}-2+p}
-{l}_{p,2,i_1,i_2}^{q_1,q_2}C_{0,i_1+1,i_2}^{q_1+3-p,{q_2}-2+p}\big)}{{\kappa_{0,-2\sigma_2}}^{-1}}. &
\ea Here the constants \(C^{q_1,q_2}_{p,i,j}\) follow equation \eqref{Cq12pij} and
\ba\label{L}
&{l}_{1,3,i_1,i_2}^{q_1,q_2}:=\frac{-(2i_1-q_1+1)({q_2}-1)_1^3 }{(i_1+1)\kappa_{{q_2},2i_2+2}\kappa_{q_1+2,2i_1+2}}, {l}_{2,3,i_1,i_2}^{q_1,q_2}
:=\frac{-2(q_2)_1^2(i_1-q_1)(i_1+1)^{-1}}{\kappa_{{q_2},2i_2+2}\kappa_{q_1+1,2i_1+2}},
{l}_{3,3,i_1,i_2}^{q_1,q_2}:=\frac{(q_1+1)({q_2}+1)(i_1+1)^{-1}}{\kappa_{{q_2},2i_2+2} \kappa_{q_1,2i_1+2}},&
\\\nonumber
&{l}_{1,2,i_1,i_2}^{q_1,q_2}:=\frac{-(q_2)_1^2(2i_1-q_1+1)(i_2-q_2) }{(i_1+1)\kappa_{{q_2+1},2i_2+2}\kappa_{q_1+2,2i_1+2}},
{l}_{2,2,i_1,i_2}^{q_1,q_2}:=\frac{-2q_2(i_1-q_1)(i_2-{q_2})(i_1+1)^{-1}}{\kappa_{{q_2+1},2i_2+2}\kappa_{q_1+1,2i_1+2}},
{l}_{3,2,i_1,i_2}^{q_1,q_2}:=\frac{(q_1+1)(i_2-{q_2})(i_1+1)^{-1}}{\kappa_{{q_2+1},2i_2+2} \kappa_{q_1,2i_1+2}}.
\ea
\end{thm}

\bpr By equations \eqref{Blik}, \eqref{nn}, and Lemma \ref{A.4}, the third component of \({\rm B}^{q_1}_{i_1,k_1}{\rm B}^{{q_2}}_{i_2,{k_2}}\) is given by
\bas
&\sum^3_{p=1} \frac{{l}_{p,3,i_1,i_2}^{q_1,q_2}{{\rm N}^{q_1+3-p}\mb{{{z}}^{i_1+1}}}{\rm N}^{{q_2}-3+p}\mb{{z}^{i_2}}}{\Delta^{-k_1-k_2}}
=\sum_{j=\max \{\sigma_2-1, 0\}}^{s_1+s_2+\lfloor\frac{r_1+r_2}{2}\rfloor}\sum^3_{p=1}\frac{{l}_{p,3,i_1,i_2}^{q_1,q_2}C_{j,i_1+1,i_2}^{q_1+3-p,{q_2}-3+p}
\N^{2j+|r_2-r_1|}\mb{{z}^{2j-\sigma_2+1+|r_2-r_1|}}}{\Delta^{{j-\sigma_1-\lfloor\frac{r_1+r_2}{2}\rfloor}}},&
\eas where \({l}_{p,3,i_1,i_2}^{q_1,q_2}\) for \(p=1, 2, 3\) are defined by equations \eqref{L}. Now using the latter and equation \eqref{Blik},
the third component of \([{\rm B}^{q_1}_{i_1,k_1},{\rm B}^{q_2}_{i_2,k_2}]= {\rm B}^{q_1}_{i_1,k_1}{\rm B}^{q_2}_{i_2,k_2}
-{\rm B}^{q_2}_{i_2,k_2}{\rm B}^{q_1}_{i_1,k_1}\) is given by
\bas
\sum_{j=\max \{\sigma_2-1, -1\}}^{s_1+s_2+{\lfloor\frac{r_1+r_2}{2}\rfloor}}
&\frac{\sum^3_{p=1}{l}_{p,3,i_2,i_1}^{q_2,q_1}C_{j,i_2+1,i_1}^{q_2+3-p,{q_1}-3+p}
-\sum^3_{p=1}{l}_{p,3,i_1,i_2}^{q_1,q_2}C_{j,i_1+1,i_2}^{q_1+3-p,{q_2}-3+p}}{(2j+|r_2-r_1|+1)(2j+1-\sigma_2+|r_2-r_1|)^{-1}
(\kappa_{2j+|r_2-r_1|,4j-2\sigma_2+2+2|r_2-r_1|})^{-1}}{\rm B}^{2j+|r_2-r_1|}_{2j-\sigma_2+|r_2-r_1|,\sigma_1-j
+{\lfloor\frac{r_1+r_2}{2}\rfloor}}\cdot\mathbf{e}_3,&
\eas when \(2j+|r_2-r_1|+1\neq0.\) The second component of \({\rm B}^{q_1}_{i,k_1}{\rm B}^{q_2}_{j,k_2}\) is given by
\bas
&\sum^3_{p=1}{l}_{p,2,i_1,i_2}^{q_1, q_2}\Delta^{k_1+k_2}{{\rm N}^{q_1+3-p}\mb{{{z}}^{i_1+1}}}{\rm N}^{{q_2}-2+p}\mb{{z}^{i_2}}&
\\&=\sum_{j=\max \{\sigma_2-1, -1\}}^{s_1+s_2+\lfloor\frac{r_1+r_2}{2}\rfloor}\sum^3_{p=1}{l}_{p,2, i_1, i_2}^{q_1, q_2}
C_{j,i_1+1,i_2}^{q_1+3-p, {q_2}-2+p}\N^{2j+|r_2-r_1|+1}\mb{{z}^{2j-\sigma_2+|r_2-r_1|+1}}\Delta^{{\sigma_1-j+\lfloor\frac{r_1+r_2}{2}\rfloor}},&
\eas where the constants \({l}_{p,2, i_1, i_2}^{q_1, q_2}\) are defined by equations \eqref{L}. Again through the equation \eqref{Blik}, the second component of
         \([{\rm B}^{q_1}_{i_1,k_1},{\rm B}^{q_2}_{i_2,k_2}]\) is given by
\ba\label{BB2}
\sum_{j=\max \{\sigma_2-1, -1\}}^{s_1+s_2+{\lfloor\frac{r_1+r_2}{2}\rfloor}}
&\frac{\sum^3_{p=1}{l}_{p,2,i_1,i_2}^{q_1,q_2}C_{j+|r_2-r_1|,i_1+1,i_2}^{q_1+3-p,{q_2}-2+p}
-\sum^3_{p=1}{l}_{p,2,i_2,i_1}^{q_2,q_1}C_{j+|r_2-r_1|,i_2+1,i_1}^{q_2+3-p,{q_1}-2+p}}{\sigma_2(2j+1-\sigma_2+|r_2-r_1|)^{-1}
(\kappa_{2j+1+|r_2-r_1|,4j-2\sigma_2+2+2|r_2-r_1|})^{-1}}{\rm B}^{2j+|r_2-r_1|}_{2j-\sigma_2+|r_2-r_1|,\sigma_1-j
+{\lfloor\frac{r_1+r_2}{2}\rfloor}} \cdot \mathbf{e}_2.&
\ea On the other hand
\bes
{{a}}_{j,i_1,i_2}^{q_1,q_2}=\frac{\sum^3_{p=1}{l}_{p,2,i_1,i_2}^{q_1,q_2}C_{j+|r_2-r_1|,i_1+1,i_2}^{q_1+3-p,{q_2}-2+p}
-\sum^3_{p=1}{l}_{p,2,i_2,i_1}^{q_2,q_1}C_{j+|r_2-r_1|,i_2+1,i_1}^{q_2+3-p,{q_1}-2+p}}{\sigma_2(2j+1-\sigma_2+|r_2-r_1|)^{-1}
(\kappa_{2j+1+|r_2-r_1|,4j-2\sigma_2+2+2|r_2-r_1|})^{-1}}, \; \hbox{ when }\; 2j+1+|r_2-r_1|\neq0.
\ees We remark that the third component of \({\rm B}^{-1}_{-\sigma_2-1,\sigma_1+1}\) is always zero and this corresponds to
the condition \(2j+1+|r_2-r_1|=0.\) Indeed, this condition occurs when \(j=-1\) and
\bes (q_1:=2s_1 \hbox{ and } q_2:=2s_2+1) \quad\hbox{ or } \quad (q_1:=2s_1+1 \hbox{ and } q_2:=2s_2).\ees
Hence, the constant \({{a}}_{-1,i_1,i_2}^{q_1,q_2}\) are derived through equation \eqref{BB2}. Now the proof is complete by
derivation of the formula for the first component of \([{\rm B}^{q_1}_{i_1,k_1},{\rm B}^{q_2}_{i_2,k_2}]\) as
\bas
&{\sum_{j=\max \{\sigma_2-1, -1\}}^{s_1+s_2+{\lfloor\frac{r_1+r_2}{2}\rfloor}}}\frac{\sum^3_{p=1}{l}_{p,1,i_1,i_2}^{q_1,q_2}C_{j,i_1+1,i_2}^{q_1+3-p,{q_2}-1+p}
-\sum^3_{p=1}{l}_{p,1,i_2,i_1}^{q_2,q_1}C_{j,i_2+1,i_1}^{q_2+3-p,{q_1}-1+p}{\rm B}^{2j+|r_2-r_1|}_{2j-\sigma_2+|r_2-r_1|,\sigma_1-j
+{\lfloor\frac{r_1+r_2}{2}\rfloor}}}{(2j+1-2\sigma_2+|r_2-r_1|)(2j+1-\sigma_2+|r_2-r_1|)^{-1}
(\kappa_{2j+2+|r_2-r_1|,4j-2\sigma_2+2+2|r_2-r_1|})^{-1}}\cdot\mathbf{e}_1,&
\eas where
\bas
&{l}_{1,1,i_1,i_2}^{q_1,q_2}:=\frac{-({q_2}+1)_1^2(2i_1-q_1+1)(2i_2-q_2+1) }{(i_1+1)\kappa_{{q_2+2},2i_2+2}\kappa_{q_1+2,2i_1+2}},
{l}_{2,1,i_1,i_2}^{q_1,q_2}:=\frac{-2(q_2+2)(i_1-q_1)(2i_2-q_2+1)}{(i_1+1)\kappa_{{q_2+2},2i_2+2}\kappa_{q_1+1,2i_1+2}},
{l}_{3,1,i_1,i_2}^{q_1,q_2}:=\frac{(q_1+1)(2i_2-q_2+1)(i_1+1)^{-1}}{\kappa_{{q_2+2},2i_2+2} \kappa_{q_1,2i_1+2}},\qquad
\eas and the equality
\bas
&{{a}}_{j,i_1,i_2}^{q_1,q_2}=\frac{\sum^3_{p=1}{l}_{p,1,i_1,i_2}^{q_1,q_2}C_{j,i_1+1,i_2}^{q_1+3-p,{q_2}-1+p}
-\sum^3_{p=1}{l}_{p,1,i_2,i_1}^{q_2,q_1}C_{j,i_2+1,i_1}^{q_2+3-p,{q_1}-1+p}}{(2j+1-2\sigma_2+|r_2-r_1|)(2j+1-\sigma_2+|r_2-r_1|)^{-1}
(\kappa_{2j+2+|r_2-r_1|,4j-2\sigma_2+2+2|r_2-r_1|})^{-1}}.&
\eas
\epr

Now we illustrate the structure constants for a few examples.

\begin{exm}\label{SampleStructure}
The following examples are computed by using a Maple program:
\bas
{[{\rm B}^6_{8,3},{\rm B}^2_{5,2}]}&=&{\frac {1152}{785213}}{\rm B}^0_{5,9}+\frac{2560}{503217}{\rm B}^2_{7,8}
-\frac {4256}{38709}{\rm B}^4_{9,7}+{\frac{1384}{1683}}{\rm B}^6_{11,6}-{\frac {35}{9}}{\rm B}^8_{13,5},
\\
{[{\rm B}^7_{6,1},{\rm B}^3_{4,1}]}&=&-{\frac {512}{429429}}{\rm B}^ 2_{2,6} +
{\frac {512}{31603}}{\rm B}^ 4_{4,5} -{
\frac {43200}{323323}}{\rm B}^ 6_{6,4} +
{\frac {528}{637}}{\rm B}^ 8_{8,3}-{
\frac {132}{35}}{\rm B} ^ {10}_{10,2},
\\
{[{\rm B}^7_{5,6},{\rm B}^6_{7,8}]}&=&-{\frac {224}{347633}}{\rm B}^3_{2,19}
+{\frac {27440}{6605027}}{\rm B}^ 5_{4,18} -{\frac {1400}{138567}}
{\rm B}^ 7_{6,17 } -{\frac {18}{299}}{\rm B}^ 9_{8,16
} +{\frac {91}{100}}{\rm B}^ {11}_{10,15}-{\frac {143}{24}}{\rm B}^{ 13}_{12,14}.
\eas
Now we remark that \bes{{\rm B}}^{l}_{i,k}=\Delta^k{{\rm B}}^{l}_{i,0}\ees for all nonnegative  integers \(l, i, k\); this is due to
the equality \({\rm N}(\Delta)=0.\) Further recall that \(\Delta\) is invariant under the \(\Sl\)-action. Yet the following equality
demonstrates the complexity of the structure constants:
\bes
{[{\rm B}^3_{5,0},{\rm B}^4_{4,0}]}=\frac{256}{297297}{\rm B}^{-1}_{ 1, 4}-\frac{512}{42471}{\rm B}^1_{3,3}+
\frac{416}{3927}{\rm B}^3_{5,2}-\frac{1312}{1881}{\rm B}^5_{ 7, 1}+\frac{10}{3}{\rm B}^7_{9,0}.
\ees
Similar to \cite[equations (3.8a)-(3.8h)]{baidersanders}, we further present some Lie brackets that they are particularly useful
for our normal form results:
\ba\nonumber
&{[{\rm B}^{0}_{0,0},{\rm  B}_{i,k}^l]}=(l-i){\rm B}^{l}_{i,k}, \qquad {[{\rm B}^1_{0,0},{\rm  B}_{i,k}^l]}= (l-2i-1) {\rm B}^{l+1}_{i,k},
\qquad{[{\rm B}^{-1}_{0,0},{\rm  B}_{i,k}^l]}=(l+1){\rm B}^{l-1}_{i,k},&
\\\label{S3}
&[{\rm B}^{-1}_{p,0},{\rm B}^{q}_{i,k}]=	\sum_{j=\max \{q-2-i-p, -1\}}^{s-1+{\lfloor\frac{r+1}{2}\rfloor}}{{a}}_{j,p,i}^{-1,q}
{\rm B}^{2j+|r-1|}_{2j-q+1+i+p+|r-1|,s-1+k-j+{\lfloor\frac{r+1}{2}\rfloor}}.&
\ea

\end{exm}
\section{Poisson algebra structure}\label{Sec3}

We consider the ring of formal power series \(\mathbb{R}[[{x},{y},{z}]]\) and define a Poisson bracket on the ring's variables by
\be \label{ppl}
\{{x},{y}\}=  {x}, \quad \{{x},{z}\}= 2{y}, \quad\{{y},{z}\} = {z}.
\ee Since \(f\) and \(g\) from \(\mathbb{R}[[{x},{y},{z}]]\) have each a unique representation as formal power series in \({x},\) \({y},\)
and \({z},\) the Leibniz rule and bilinearity of the Poisson bracket are sufficient to uniquely determine Poisson structure for all elements
 in \(\mathbb{R}[[{x},{y},{z}]]\). In particular, the Poisson bracket is independent of the splitting functions in \(\mathbb{R}[[{x},{y},{z}]],\) i.e.,
\ba\label{poissonr}
\{f, g(hk)-(gh)k\}=0.
\ea
Indeed, by the Leibniz rule we have
\bas
&\{f, g (hk)\}- \{f, (gh) k\} = \{f, g\} hk+gk \{f, h\}+gh \{f, k\}-gh \{f, k\}-kg \{f, h\}-kh \{f, g\}= 0. &
\eas
Hence, the structure constants associated with monomials are given by
\ba\label{poissongenerat}
&\left\{ {{x}}^{{ i}}{{y}}^{j}{{z}}^{k},{{x}}^{m}{
{y}}^{n}{{z}}^{p} \right\} = \left( { i}n+jp-kn-jm
 \right) {{x}}^{{ i}+m}{{y}}^{n+j-1}{{z}}^{k+p}+ 2\left( { i}p-km \right) {{x}}^{{ i}+m-1}{{y}}
^{n+j+1}{{z}}^{k+p-1},&
\ea for arbitrary nonnegative integers \(m, n, p, i, j, k.\) Now define
\ba \label{pb}
\mathfrak{b}_{i,k}^l:=- \frac{\ad^{l+1}_{{x}}\mb{{{z}}^{i+1}\Delta^k}}{(i+1)\kappa_{l+1, 2i+2}}, \qquad \hbox{ for } -1\leq l\leq i+1, \hbox{ and } i, k\in \mathbb{N}_0,
\ea where \(\ad_{{x}} f:=\{{x}, f\}\) and \(\ad_{{x}}^n f:=\{{x}, \ad_{{x}}^{n-1}f\}\) for \(f\in \mathbb{R}[[x, y, z]]\) and \(n>1.\)

\begin{cor}\label{poisson}
The following formulas provide two alternative polynomial expansions for each \(\mathfrak{b}_{i,k}^{l}\) in terms of \(x,\) \(y,\) \(z\) and \(\Delta\):
\ba\label{bfrac}
&\mathfrak{b}_{i,k}^{2s}
=\sum_{j=0}^s\frac{-\eta^{2s+1,i+1}_j}{(i+1)\kappa_{2s+1,2i+2}}
{{x}}^{s-j}{{y}}{{z}}^{i-s-j}\Delta^{k+j},\quad
\mathfrak{b}_{i,k}^{2s}=\sum_{j=0}^{s}\frac{-\zeta^{2s+1,i+1}_j}{(i+1)\kappa_{2s+1,2i+2}} {{x}}^
{s-j}{{y}}^{2j+1}{{z}}^{i-j-s}\Delta^k, &
\\\nonumber
&\mathfrak{b}_{i,k}^{2s+1}= \sum_{j=0}^{s+1}\frac{-\eta^{2s+2,i+1}_j}{(i+1)\kappa_{2s+2,2i+2}}
{{x}}^{s-j+1}{{z}}^{i-s-j}\Delta^{k+j}, \quad
\mathfrak{b}_{i,k}^{2s+1}=\sum_{j=0}^{s+1}\frac{-\zeta^{2s+2,i+1}_j}{(i+1)\kappa_{2s+2,2i+2}}
{{x}}^{s-j+1}{{y}}^{2j}{{z}}^{i-j-s}\Delta^k,&
\ea where
\bas
&{\eta^{2s+1,i+1}_j}:=\frac{( -1) ^{j}(s)_{-1}^j(i+1)_{-1}^{s+j+1}(2i+1)_{-2}^{s}{2}^{s+j+1}}{(j)! (2i+2)_{-2}^{j}}, \quad
\zeta^{2s+1,i+1}_j:={\frac {(i+1)!(2s+1)! {2}^{2j+1}}{(s-j)! (2j+1)! (i-j-s)!}},&
\\
&{\eta^{2s+2,i+1}_j}:=\frac{( -1) ^{j}(s+1)_{-1}^j(i+1)_{-1}^{s+j+1}(2i+1)_{-2}^{s+1}{2}^{s+j+1}}{(j+1)! (2i+2)_{-2}^{j}},
\quad
\zeta^{2s+2,i+1}_j:={\frac { (i+1)!  ( 2s +2)!  {2}^{2j}}{ ( s-j+1)! ( 2j )! ( i-j-s) !}}.&
\eas
\end{cor}
\bpr
Due to the previous lemma, the actions of \(\ad_{{x}}\) and \(\ad_{\N}\) on \({{z}}^{i}\) are identical. Hence, our claim readily follows from Lemma \ref{A.1}.
\epr
Now we define a vector space \(\mathfrak{B}\) as
\be\label{Bfrac}
\mathfrak{B}:=\Span\{\mathfrak{b}_{0,0}^1+\sum\beta^l_{i,k}\mathfrak{b}_{i,k}^l \big|  -1\leqslant l \leqslant 2i+1,  i,k \in\mathbb{N}_0 , \beta^l_{i,k}\in{\mathbb{R}}\}.
\ee The following two lemmas show that \(\mathfrak{B}\) is a Poisson algebra and it is Lie-isomorphic to \(\B.\)

\begin{lem}\label{iso}
The space \(\mathfrak{B}\) is invariant under the Poisson bracket and the linear map
\begin{eqnarray}
\label{fiso}
\Psi: \mathfrak{B} &\rightarrow&\mathscr{B}, \quad \hbox{ defined by }\quad
\Psi(\mathfrak{b}_{i,k}^{l})={\rm B}_{i,k}^{l},
\end{eqnarray} is a Lie isomorphism.
\end{lem}
\begin{proof} By the Leibniz rule we have
\bas
\ad^{n}_{{x}} \mb{{z}^{i+1}} &=&
{z}\,\ad^{n}_{{x}} \mb{{z}^{i}}+ n\,\ad_{{x}}\mb{ {z}}\,\ad^{n-1}({x})\mb{{z}^{i}}+ \frac{n(n-1)}{2!}\,\ad^2_{{x}}\mb {{z}}\,\ad^{n-2}_{{x}}\mb{{z}^{i}}
\\
&=&{z}\,\ad^{n}_{{x}}\mb{{z}^{i}}+ 2 n{y}\,\ad^{n-1}_{{x}}\mb{{z}^{i}}+ n(n-1) {x}\,\ad^{n-2}_{{x}} \mb{{z}^{i}}.
\eas Due to equation \eqref{fiso}, we have \(\Psi({x})=\N,\) \(\Psi({y})=\frac{{\rm H}}{2},\) and \(\Psi({z})=-\M.\) Further,
\bes
\ad^{n}_{{x}} {{z}}^{i+1}\mb{\Delta^k}=\ad^{n}_{{x}}\mb{ {{z}}^{i+1}}\Delta^k.
\ees The actions of \(\ad^{n}_{{x}}\) on \({{z}}^{i+1}\Delta^k\) and \(\ad_{\N}\) on \({{z}}^{i}\Delta^k\M\) are identified through \(\Psi.\)
 Then, the proof follows an induction on \(n,\) structure constants \eqref{ppl} and those governing the \(\Sl\) triple \(\M,\N,\) and \({\rm H}.\)
\end{proof}
Now we present a ring structure constants for \(\mathfrak{B}\) so that \(\mathfrak{B}\) is a Poisson algebra.

\begin{lem}
The space \(\mathfrak{B}\) is a Poisson algebra. In particular, let \(q_1=2s_1+r_1\) and \(q_2=2s_2+r_2.\) Then, the ring structure constants are given by
\bas
&\mathfrak{b}_{{i_1-1},k_1}^{q_1-1}\mathfrak{b}_{{i_2-1},k_2}^{q_2-1}=
\frac{-1}{{i_1i_2\kappa_{q_1, 2{i_1}}\kappa_{q_2, 2{i_2}}}}	\sum_{p=\max\{\sigma_2,0\}}^{s_1+s_2+\lfloor\frac{r_1+r_2}{2}\rfloor}
\frac{\kappa_{2p+|r_2-r_1|, 2(2p-\sigma_2+|r_2-r_1|)}C_{p,i,j}^{q_1,q_2}}{(2p-\sigma_2+|r_2-r_1|)^{-1}}
{\mathfrak{b}}^{2p+|r_2-r_1|-1}_{2p-\sigma_2+|r_2-r_1|-1,\sigma_1-p+{\lfloor\frac{r_1+r_2}{2}\rfloor}}.&
\eas
\end{lem}
\bpr The proof directly follows from \eqref{pb} and the formulas given in  Theorem \ref{A.3}. Indeed, we have
\(\mathfrak{b}_{{i_1-1},k_1}^{q_1-1}\mathfrak{b}_{{i_2-1},k_2}^{q_2-1}=\frac{\N^{q_1}\mb{{z}^i}\N^{q_2}\mb{{z}^j}\Delta^{k_1+k_2}}
{i_1i_2\kappa_{q_1, 2{i_1}}\kappa_{q_2, 2{i_2}}}\) and
\bas
&
\frac{1}{{i_1i_2\kappa_{q_1, 2{i_1}}\kappa_{q_2, 2{i_2}}}}	\sum_{p=\max\{\sigma_2,0\}}^{s_1+s_2+\lfloor\frac{r_1+r_2}{2}\rfloor}C_{p,i,j}^{q_1,q_2}\N^{2p+|r_2-r_1|}
\mb{{z}^{2p-\sigma_2+|r_2-r_1|}}\Delta^{\sigma_1-p+{\lfloor\frac{r_1+r_2}{2}\rfloor}}&
\\
&=\frac{-1}{{i_1i_2\kappa_{q_1, 2{i_1}}\kappa_{q_2, 2{i_2}}}}\sum_{p=\max\{\sigma_2, 0\}}^{s_1+s_2+\lfloor\frac{r_1+r_2}{2}\rfloor}
\frac{\kappa_{2p+|r_2-r_1|, 2(2p-\sigma_2+|r_2-r_1|)}C_{p,i,j}^{q_1,q_2}}{(2p-\sigma_2+|r_2-r_1|)^{-1}}
{\mathfrak{b}}^{2p+|r_2-r_1|-1}_{2p-\sigma_2+|r_2-r_1|-1,\sigma_1-p+{\lfloor\frac{r_1+r_2}{2}\rfloor}}.&
\eas
\epr The following theorem presents a property that is similar to the Hamiltonian cases, \ie the rate change of functions
along with vector fields from \(\B\) can be computed by the Poisson bracket.
\begin{thm}\label{thmBpoisson}
For each \(l, i, k\), we have
\ba\label{Bpoisson}
&{\rm B}^l_{i,k}=\sum_{j=1}^3 \{x_j, \Psi^{-1}( {\rm B}^l_{i,k})\}\,{\bf e}_j, \quad \hbox{ where } \quad x_1:=x, x_2:= y, x_3:= z.&
\ea This representation for \(B^l_{i, k}\) indicates that the family of vector fields from \(\B\) and their associated dynamics
are uniquely determined by their secondary Clebsch potentials. Furthermore, the change rate of any formal power series in
\((x, y, z),\) say \(F: \mathbb{R}^3\rightarrow \mathbb{R},\) along a vector filed \(v\) from \(\B\) is given by
\be\label{PoissonForm}
\frac{d F}{d t}:= \{F, \Psi^{-1}(v)\}.
\ee
\end{thm}
\bpr
From equations \eqref{Blik}, \eqref{fiso} and \eqref{pb}, we have
\bas
{\rm B}^{l}_{i,k}&=&{(l-2i-1)\mathfrak{b}_{i,k}^{l+1} }\frac{\D}{\D {x}}+{(l-i)\mathfrak{b}_{i,k}^{l}}\frac{\D}{\D {y}}+{(l+1)\mathfrak{b}_{i,k}^{l-1}}\frac{\D}{\D {z}}.
\eas Then, the proof follows the Lie isomorphism \eqref{fiso}, the formulas \eqref{S3},
\(\Psi({x})=\N,\) \(\Psi({y})=\frac{{\rm H}}{2},\) \(\Psi({z})=-\M,\) \({\rm B}^1_{0,0}=-{\rm N}\) and \({\rm B}^{-1}_{0,0}=-{\rm M}.\)
The vector field representation \(v\) from \(\B\) in Poisson bracket
form \eqref{PoissonForm} directly follows from equation \eqref{Bpoisson}, the linearity and the continuity (in filtration topology) of
the Lie isomorphism \(\psi,\) the continuity and bilinearity of the Poisson bracket, and finally, the chain and Leibniz rules.
\epr

\section{Geometrical features of integrable solenoidal vector fields }\label{Sec4}

\begin{defn}\label{mechanic}
A vector field \(v\) is called {\it solenoidal} (nondissipative, incompressible, or volume-preserving) when \({\rm div}(v(x))=0,\) and
otherwise the vector field \(v\) is called {\it generalized dissipative}, \ie \({\rm div}(v(x))\neq0\). When \(v(x)= \nabla f(x)\) for a scalar function \(f(x),\)
the vector field \(v\) is called a {\it gradient} or a {\it globally potential} vector field. Examples of this are the gravitational
potential, a mechanical potential energy, and the electric potential energy. The vector field \(v(x)\) is said to be nonpotential (non-gradient)
when there exists at least a point \(x\in{\mathbb{R}}^3\) such that \({\bf curl}(v(x))\neq 0\); \eg see \cite[page 1]{tarasov2008quantum}.
\end{defn}

\begin{thm}\label{Solenoid}
For every \(v\in {\mathscr{B}},\) \(v\) is solenoidal.
\end{thm}
\bpr
By the Leibniz rule and \({{\rm B}}^{l}_{i,k}=\Delta^k{{\rm B}}^{l}_{i,0}\), we observe that
\ba\label{p0}
\nabla\cdot{\rm B}^l_{i,k}={\nabla} \cdot\Delta^k{\rm B}^l_{i,0}=\Delta^k{\nabla} \cdot{\rm B}^l_{i,0}
+{\nabla}(\Delta^k)\cdot{\rm B}^l_{i,0}.
\ea We claim that \({\nabla} \cdot{\rm B}^l_{i,0}=0\) and \({{\nabla}(\Delta^k)}\cdot{\rm B}^l_{i,0}=0.\) Using Lemma \ref{A.4}, the partial derivatives of
\({\N}^{l}\mb{{{z}}^{i+1}}\) are given by
\ba\label{deriv}
&\frac{\D}{\D{{x}}}\N^{l+2}\mb{{{z}}^{i+1}}=( l+1 )( l+2 )(i+1)\N^{ l}\mb{{z}^{i}}, \;\;
\frac{\D}{\D{{y}}}\N^{l+1}\mb{{{z}}^{i+1}}
=2(l+1)(i+1)\N^{l}\mb{{{z}}^{i}}, \;\;
\frac{\D}{\D {{z}}}{\N}^{l}\mb{{{z}}^{i+1}}=(i+1)\N^{l}\mb{{{z}}^{i}}.\quad&
\ea
Equation \eqref{Blik} and Lemma \ref{A.4} give rise to
\ba\label{p1}
&{\nabla} \cdot{\rm B}^l_{i,0}=(l+1)(i+1)\left(\frac{( l+2 )(2i-l+1)}{\kappa_{l+2,2i+2}}
+\frac{2(i-l)}{\kappa_{l+1,2i+2}}-\frac{1}{\kappa_{l,2i+2}}\right)\N^{ l}\mb{{z}^{i}}=0.&
\ea
On the other hand for \(l=2s,\) equation \eqref{Blik}  gives rise to
\bas
&{{\nabla}(\Delta^k)}\cdot{\rm B}^l_{i,0}=\frac{k(2i-s+1) {z}\Delta^{k-1}\N^{2s+2}\mb{{{z}}^{i+1}}}{(i+1)\kappa_{s+2,2i+2}}
-\frac{2k(i-s){y} \Delta^{k-1}\N^{2s+1}\mb{{{z}}^{i+1}}}{(i+1)\kappa_{s+1,2i+2}}
- \frac{k(s+1){x}\Delta^{k-1}\N^{2s}\mb{{{z}}^{i+1}}}{(i+1)\kappa_{s,2i+2}}\qquad\qquad\qquad\quad&
\\&\qquad\quad=2k\,i!(2s+1)!\sum_{j=0}^{s} \frac{\big({{(s+1)(i-j-s+1)}}-{{j(i-2s)}}-{{(i-s+1)(s-j+1)}}\big){{x}}^{s-j+1}{{y}}^{2j}{{z}}^{i-j-s+1}
\Delta^{k-1}}{{2}^{-2j}(2i-2s+1)!^{-1}(2j)!(s-j+1)!(i-j-s+1)!}.&
\eas The last equality is derived from Lemma \ref{A.1}. Hence, \({\nabla}(\Delta^k)\cdot{\rm B}^l_{i,0}=0\) due to
\be
\label{invd}
{{(s+1)(i-j-s+1)}}-{{j(i-2s)}}-{{(i-s+1)(s-j+1)}}=0.
\ee When \(l=2s+1,\) \({{\nabla}(\Delta^k)}\cdot{\rm B}^l_{i,0}\) is given by
\bas
&\frac{k(2i-2s) {z}\Delta^{k-1}\N^{2s+3}\mb{{{z}}^{i+1}}}{(i+1)\kappa_{2s+3,2i+2}}
-\frac{2k(i-2s-1){y} \Delta^{k-1}\N^{2s+2}\mb{{{z}}^{i+1}}}{(i+1)\kappa_{2s+2,2i+2}}
- \frac{k(2s+2){x}\Delta^{k-1}\N^{2s+1}\mb{{{z}}^{i+1}}}{(i+1)\kappa_{2s+1,2i+2}}&
\\&=\sum_{j=0}^{s} \frac {k(2s+1)!(i+1)!(2i-1-2s)!{x}^{s+1-j}{y}^{2j+1}{z}^{i-j-s}}{\Delta^{1-k}{2}^{-2j}(i-j-s-1)!(s-j)!(2i+2)!(2j)!(i+1)}
\big({\frac{4(i-s)(2s+2)_1^2}{(s+1-j)(2j+1)}}-{\frac{4(s+1)(2i-2s)_{-1-i}^2}{(s+1-j)(i-j-s)}}-{\frac{4(s+1)(2i-2s)_1^2}{(2j+1)(i-j-s)}}\big)\quad\quad&
\\&+\frac{ki! {y}^{2s+3}{z}^{i-2s-1}\Delta^{k-1}}{(2i+2)!}\big({\frac{(2i-2s){2}^{2s+3}(2i-1-2s)!}{(i-2s-2)!}}-{\frac {(i-2s-1){2}^{2s+3} (2i-2s)!}{ (i-2s-1)!}}\big).
\qquad\qquad\qquad\qquad\qquad\qquad&
\eas Hence \({{\nabla}(\Delta^k)}\cdot{\rm B}^l_{i,0}=0\) due to the equation
\ba\label{Eq4}
&{\frac{(i-s)(2s+2)_1^2}{(s+1-j)(2j+1)}}-\frac{(s+1)(2i-2s)_{-1-i}^2}{(s+1-j)(i-j-s)}-\frac{(s+1)(2i-2s)_1^2}{(2j+1)(i-j-s)}=
\frac{2(i-s)(2i-1-2s)!}{(i-2s-2)!}-\frac{(i-2s-1)(2i-2s)!}{(i-2s-1)!}=0.&
\ea
Equations \eqref{Eq4}, \eqref{p1}, \eqref{invd} and \eqref{p0} conclude the proof.
\epr

\begin{thm}\label{int}
Let \(i\) and \(k\) be arbitrary nonnegative integers and \(-1\leq l\leq i+1.\)
\begin{itemize}
\item\label{FirstInts} Polynomials \(\mathfrak{b}^l_{i,0}\) and \(\Delta\) are two first integrals for \({\rm B}^l_{i,k},\) \ie
\bas
&{\rm B}^l_{i,k}(\mathfrak{b}^l_{i,0})=0\quad\hbox{ and }\quad {\rm B}^l_{i,k}(\Delta)=0.&
\eas Indeed for every \(v\in \B,\) \(\Psi^{-1}{(v)}\in \mathfrak{B}\) and \(\Delta\) are two first integrals for \(v.\)
\item A Clebsch potential representation for \({\rm B}^l_{i,k}\) is given by
\ba\label{Clebsch}
{\rm B}^l_{i,k}=\Delta^k(\nabla  {\mathfrak{b}^l_{i,0}} \times\nabla \Delta).
\ea Equation \eqref{Clebsch} provides an alternative representation for each vector field \(v\) in \(\B\) by using the primary and secondary Clebsch potentials
\(\Delta\) and \(\Psi^{-1}{(v)}\in \mathfrak{B}.\)
\item The polynomial functions \({\mathfrak{b}^l_{i,0}}\) and \(\Delta\) are two functionally independent first integrals for \({\rm B}^l_{i,0}\).
\item The ring of invariants for \({\rm B}^{-1}_{i,k},\) \({\rm B}^{i}_{i,k},\) and \({\rm B}^{2i+1}_{i,k}\) includes
\(\left\langle{\Delta, z}\right\rangle,\) \(\left\langle{\Delta, y}\right\rangle\) and \(\left\langle{\Delta, x}\right\rangle,\) respectively.
\end{itemize}
\end{thm}
\bpr
By equation \eqref{pb}, Leibniz rule and \({\rm N}z= \{x, z\}= 2y,\) we have \(\mathfrak{b}_{i,0}^l:=- \frac{{\rm N}^{l+1}\mb{{{z}}^{i+1}}}{(i+1)\kappa_{l+1, 2i+2}}.\)
Hence, equation \eqref{Blik} and Lemma \ref{A.4} imply
\bas{\rm B}^l_{i,k}(\mathfrak{b}^l_{i,0})
&=&\frac{(2i+1-l)!\big(l{\rm N}^{l+2}\mb{{{{z}}^{i+1}}}{\rm N}^{l-1}\mb{{{z}}^{i}}+2(i-l){\rm N}^{l+1}\mb{{{z}}^{i+1}}{\rm N}^{l} \mb{{{{z}}^{i}}}
-(2i+2-l){\rm N}^{l}\mb{{{z}}^{i+1}}{\rm N}^{l+1} \mb{{{{z}}^{i}}}\big)}{-\Delta^{-k}(l+1)^{-1}\left(i+1\right)\left(2i+2\right)!}.
\eas Now we claim that
\ba\label{SC}
l{\rm N}^{l+2}\mb{{{{z}}^{i+1}}}{\rm N}^{l-1}\mb{{{z}}^{i}}+2(i-l){\rm N}^{l+1}\mb{{{z}}^{i+1}}{\rm N}^{l} \mb{{{{z}}^{i}}}-(2i+2-l){\rm N}^{l}
\mb{{{z}}^{i+1}}{\rm N}^{l+1} \mb{{{{z}}^{i}}}=0.
\ea
Equality \eqref{SC} is trivial for the case \(l=0.\) Let \(l\neq0,\) \(l:=2s+r,\) \(r=0\) and \(1.\) By equation \eqref{Nqq1Nqq1}, we have
\bas
&{\rm B}^l_{i,k}(\mathfrak{b}^l_{i,0})= \sum_{n=0}^{l}f_r(n){{{x}}^{l-n}{{z}}^{2{i}-l-n}{{y}}^{2n+1}},&
\eas where \(f_r(n):=\sum_{p=0}^n F_r(n,p)\) for all \(0\leqslant n\leqslant l,\)
\bas
&F_{0}(n,p):={\frac{(2s)!{i}!({i}+1)!(2s+1)!{2}^{2n+2}({
i}+p-n-s-1)!^{-1}}{(2p)!(s-p)!({i}-p-s)!(s-n+p-1)!(2n-2p)!}}\big({\frac{(s+1)({i}+p-n-s)^{-1}}{(s-p+1)(2n-2p+1)}}+{\frac{(i-2s)({i}+p-n-s)^{-1}}{(2p+1)(s+p-n)}}
-{\frac{(i-s+1)({i}-p-s+1)^{-1}}{(2n-2p+1)(s+p-n)}}\big) &
\eas and \(F_{1}(n,p)\) is given by
\bas
&\frac{{2}^{2n+1}i!(i+1)!(2s+1)!(2s+2)!(i-n+p-s-1)!^{-1}}{(s-n+p)!(2p)!(s-p)!(2n-2p)!(i-p-s-1)!}
\Big(\frac{(2s+3)(i-n+p-s)^{-1}}{(s+1-p)(2p+1)}+\frac{2(i-2s-1)(i-p-s)^{-1}}{(s+1-p)(2n-2p+1)}-\frac{(2i+1-2s)(i-p-s)^{-1}}{(2p+1)(s+1-n+p)}\Big).&
\eas
Now we follow Zeilberger's algorithm \cite [Chapter 6]{petkovsek2006b} to prove \(f_r(n)=0.\) Let
\bas
&G_0(n,p):=\frac{p\left(2s-i\right)\left(2ip+2{s}^{2}+5p-2is-2i-2n-5\right)}{\left(2n-2p+3\right)\left(n-p+1\right)(2p)!
\left(s-p+1\right)!\left(2n-2p+1\right)!\left(i-p-s+1\right)!\left(s+p-n-1\right)!\left(i+p-n-s-1\right)!}.&
\eas Then,
\ba\label{sumZe}
&-2(2n+3)F_0(n+1,p)+(n-i-1)F_0(n,p)=G_0(n,p+1)-G_0(n,p).&
\ea
Next, we add both sides of the equality \eqref{sumZe} over \(p\) for all \(0\leqslant p\leqslant n-1.\) Hence \(G_0(n,n)\) is given by
\bas
&(n-i-1)f_0(n)-(4n+6)f_0(n+1)+(4n+6)\big(F_0(n+1, n)+F_0(n+1, n+1)\big)-(n-i-1)F_0(n,n).&
\eas On the other hand \(G_0(n,n)= 2(2n+3)\big(F_0(n+1, n)+F_0(n+1, n+1)\big)-(n-i-1)F_0(n,n)\) and
\bas
&-2(2n+3)f_0(n+1)+(n-i-1)f_0(n)=0.&
\eas
Thereby, \(f_0(n)=\frac{f_0(0)}{2^n} \prod_{j=0}^{n-1} \frac{j-i-1}{2j+3}\); also see \cite[page 103]{petkovsek2006b}.
Since
\bas
&f_0(0)=F_0(0, p)=\frac{\frac{1}{i-s}+\frac{i-2s}{s(i-s)}-\frac{1}{s}}{\left(s-1\right)!s!\left(i-s\right)!\left(i-s-1\right)!}=0,&
\eas \(f_0(n)=0\) for any \(n.\) Now let \(l=2s+1,\) \ie \(r:=1.\) Let
\bas
&G_1(n,p)=\frac{p{2}^{2n+2}2i!(i+1)!(2s+1)!(2s+2)!(2s+3)(2i+1-2s)(i-2s-1)}{(i-p-s)(2n-2p+1)(s+1-p)(n-p+1)(3+2n-2p)(i-p-s-1)!(2n-2p)!(s-p)!(2p)!(i-n+p-s-1)!(s-n+p)!}&
\eas and thus,
\bas
(-3-2n)F_1(n+1,p)+2(n-i-1)F_1(n,p)=G_1(n,p+1)-G_1(n,p).
\eas
Hence,
\bas
-(2n+3)f_1(n+1)+2(n-i-1)f_1(n)=0,
\eas \(f_1(0)=F_1(0,0)=0\) and finally, \(f_1(n)=0.\) Hence, \({\rm B}^{l}_{i,k}(\mathfrak{b}^l_{i,0})=0\) and \(\mathfrak{b}^l_{i,0}\) is an invariant function
for \({\rm B}^{l}_{i,k}.\) Equation \eqref{p0} and Theorem \ref{Solenoid} give rise to \({\rm B}^{l}_{i,k}(\Delta)={\nabla}(\Delta)\cdot{\rm B}^l_{i,k}=0.\) Hence,
\(\Delta\) is also a first integral for \({\rm B}^{l}_{i,k}\).

By Lemma \ref{A.4},
\bas
&\nabla{\mathfrak{b}^l_{i,0}}=\frac{1}{{{\kappa_{l+1, 2i+2}}}}\Big({{ l(l+1)\N^{l-1}\mb{{{z}}^{i}}},{2(l+1) \N^{l}\mb{{{z}}^{i}}},{ \N^{l+1}\mb{{{z}}^{i}}}}\Big).&
\eas
Since \({\rm B}^l_{i,k}\) is tangent to the level surfaces of \(\mathfrak{b}^l_{i,0}\) and \(\Delta\) for any \((x, y, z)\in \mathbb{R}^3,\) there exists a
function \(S^l_{i, k}(x, y, z)\) such that
\bas
&{\rm B}^l_{i, k}=S^l_{i, k}\;\nabla  {\mathfrak{b}^l_{i,0}} \times\nabla \Delta.&
\eas Hence,
\ba\label{EE3}
&(S^l_{i, 0}\,\nabla{\mathfrak{b}^l_{i,0}}\times\nabla\Delta)\cdot{{\bf e}_3}-{\rm B}^{l}_{i, 0}\cdot{{\bf e}_3}=
-S^l_{i, 0}\,\frac {l{y}{ \N}^{l-1}\mb{{{z}}^i} +{z}\N^l\mb{{{z}}^i}}{{\kappa}_{ l+1,2\,i+2}}+{\frac{\N^l\mb{{{z}}^{i+1}}}{ 2( i+1) {\kappa}_{
l,2\,i+2 }}}=0.&
\ea
Therefore by equation \eqref{Nq1},
\bas
&\frac{{{\sum_{n=0}^{s}\zeta^{2s,i+1}_n }{}}{{x}}^{s-n}{{y}}^{2n}{{z}}^{i-n-s+1}}{2( i+1) {\kappa}_{l,2i+2 }}
-\frac{lS^l_{i, 0}\sum_{n=1}^{s}\zeta^{2s-1,i}_{n-1} {{x}}^{s-n}{{y}}^{2n}{{z}}^{i-n-s+1}}{{{\kappa}_{ l+1,2i+2} }}-\frac{S^l_{i, 0}
\sum_{n=0}^{s}\zeta^{2s,i}_n {{x}}^{s-n}{{y}}^{2n}{{z}}^{i-n-s+1}}{{{\kappa}_{ l+1,2i+2} }}=0.&
\eas
Thus, \(S^l_{i, 0}=1\) due to
\ba\label{equ}
&\frac{l\zeta^{2s-1,i}_{n-1}+\zeta^{2s,i}_n}{{\kappa}_{ l+1,2i+2} }-\frac{\zeta^{2s,i+1}_n}{ 2( i+1) {\kappa}_{
l,2i+2 }}=0,&
\ea for all \(0\leqslant n\leqslant s,\) and \(\zeta^{2s-1,i}_{-1}=0.\) The condition \(S^l_{i, 0}=1\) implies \(S^l_{i, k}=\Delta^k.\)

Since \(B^l_{i,0}\neq0\) for almost everywhere, the last two claims are immediately concluded from the first and second claim, and Lemma \ref{common}.
\epr

Define the grading function
\ba  \label{grade}
\delta({\rm B}^{l}_{i,k}):={i+2k},
\ea that is, the standard degree of homogeneous vector fields minus one. This makes the space \((\B, [\,-,-\,])\) a graded Lie algebra.
Hence, for \(N\in \mathbb{N}_0\) the vector space
\bas
\B_N:=\Span\{{\rm B}^{l}_{N-2k,k}: l=-1,\ldots ,2(N-2k)+1, k=0,\ldots ,\lfloor\frac{N}{2}\rfloor\},
\eas consists of all \(\delta\)-homogenous vector fields of grade \(N.\) For \(v\in {\B},\) we define (also see \cite[Definition 3.2]{GazorKazemi})
\({\rm Terms}(v):= \cup^3_{p=1}{\rm Terms}\,(v\cdot \mathbf{e}_p),\) while
\bas &{\rm Terms}\,(v\cdot \mathbf{e}_p):= \{\hbox{All monomials contributing in Taylor expansion}\, v\cdot {{\bf e}_p}\}.&
\eas For an instance we have \({\rm Terms}(2x^2+3xy+5)=\{1, x^2, xy\}.\) Thereby, for any two \(\delta\)-grade homogeneous vector fields
\(v_{1}\) and \(v_2,\) where \(\delta(v_1)\neq\delta(v_2),\)
\bas
&{\rm Terms}(v_1\cdot {{\bf e}_p})\cap {\rm Terms}(v_2\cdot {{\bf e}_p})=\varnothing, \,\,\,\hbox{for any}\,\, p=1,2,3.&
\eas When \(i, k\in\mathbb{N}_0,\) \(-1\leqslant l\leqslant 2i+1,\) and \(N:=i+2k,\) we define a condition for a nonnegative integer \(m\) by
\ba\label{Restricted}
&0\leqslant m\leqslant \min\left\{{\lfloor\frac{N}{2}\rfloor, {\lfloor\frac{2N-2k-l+1}{2}\rfloor}, {\lfloor\frac{2k+l+1}{2}\rfloor}}\right\}&
\ea and next, a set \(P^{l}_{i,k}\) by
\bas
&P^{l}_{i,k}:=\{(m_1, m_2, m_3): m_1=l+2(k-m_3), m_2=N-2m_3, \hbox{the condition \eqref{Restricted} holds for } m_3\}.&
\eas

\begin{lem}\label{common}
Let \({\rm B}^{l}_{i,k}\in {\B},\) \((m_1, m_2, m_3)\in P^l_{i, k},\) and \(p\in\{1, 2, 3\}\). Then,
\begin{enumerate}
\item\label{a} \({\rm Terms }({\rm B}^{m_1}_{m_2,m_3}\cdot {{\bf e}_p})\neq \varnothing \) when \(l\neq (3-p)i+2-p.\)
Otherwise, \({\rm Terms }({\rm B}^{m_1}_{m_2,m_3}\cdot {{\bf e}_p})= \varnothing,\) \ie for \((p,l)=(1,2i+1),\) \((p,l)=(2,i)\) and \((p,l)=(3,-1).\)
\item\label{b} When \({\rm Terms }({\rm B}^{l}_{i,k}\cdot {{\bf e}_p})\neq\varnothing,\)
\bes
{\rm Terms }({\rm B}^{m_1}_{m_2,m_3}\cdot {{\bf e}_p})\subseteq{\rm Terms }({\rm B}^{l}_{i,k}\cdot {{\bf e}_p})\,\,\hbox{and} \,\, P^{l}_{i,k}=P^{m_1}_{m_2,m_3}.
\ees
\item\label{c} Let \({\rm Terms }({\rm B}^{l}_{i,k})={\rm Terms }({\rm B}^{l'}_{i',k'}).\) Then, natural numbers \(l-l'\) and \(i-i'\) are even. Furthermore, the
inequalities \(k\neq k', l\neq l'\) and \(i\neq i'\)  are equivalent. In particular, \(k=k'\) implies \((l,i,k)= (l',i',k').\)
\item\label{d} The set
\bas
\{{\rm Terms}({\rm B}^{-1}_{N,0}\cdot {{\bf e}_p}), {\rm Terms}({\rm B}^{N}_{N,0}\cdot {{\bf e}_p}),
{\rm Terms}({\rm B}^{2N+1}_{N,0}\cdot {{\bf e}_p}), {\rm Terms}({{\rm B}}^{l}_{N,0})\; \hbox{ for }\; 0\leq l \leq 2N, l\neq N\}
\eas is a partition for
\bas
\cup\{{\rm Terms}(v_N\cdot {{\bf e}_p}), v_N\in \B\}.
\eas
\item\label{e} \({\rm Terms}(\mathfrak{b}^{l}_{i,k})={\rm Terms}(\mathfrak{b}^{m_1}_{m_2,m_3})\neq\varnothing.\)
\item\label{f}
For any \(0\neq v\in{\mathscr{B}}\) and nonnegative integer \(N,\) there exists a unique polynomial vector \(v^{j}_{N-2k,0}\in
{\mathbb{R}}\{{\rm B}^{j}_{N-2k,0}\}\) for each \(-1\leqslant j\leqslant 2N-4k+1\) and \(0\leqslant
k\leqslant \min\{\lfloor\frac{2N-j+1}{4}\rfloor, \lfloor \frac{N}{2}\rfloor\}\) so that
\ba\label{vterm}
&v=\sum_{N=0}^{\infty}\sum^{2N+1}_{j=-1}\sum_{k=0}^{\min\{\lfloor\frac{2N-j+1}{4}\rfloor, \lfloor \frac{N}{2}\rfloor\}} v^{j}_{N-2k,0}\Delta^k.&
\ea
Furthermore, \(v=0\) if and only if \( v^{j}_{N-2k,0}=0\) for all \(j, k, N.\)
\end{enumerate}
\end{lem}

\bpr
The claim in part \ref{a} directly follows from equation \eqref{Blik}. For claim \ref{b}, let \({\rm B}^{l'}_{i',k'}\) be the vector
field in \(\B\) with \((l,i,k)\neq (l',i',k').\) Here we only consider the case \(p=3.\) Using the polynomial expansion for
\({\Delta}^{k}, {\Delta}^{k'},\) Lemma \ref{A.1} and definition \eqref{bm}, the monomials appearing in \({\rm B}^{l}_{i,k}\) and \({\rm B}^{l'}_{i',k'}\) follow
\bas
&P(x, y, z) ={x}^{s-n_1+p_1}{y}^{2n_1+r+2(k-p_1)}{z}^{i-n_1-s-r+p_1}, Q(x, y, z)={x}^{s'-n_2+p_2}{y}^{2n_2+r'+2(k'-p_2)}{z}^{i'-n_2-s'-r'+p_2},&
\eas
for some  \(n_1, n_2, p_1, p_2,\) where \(l=2s+r\) and \(l'=2s'+r'.\) Let \(P= Q.\) Then,
\ba
\label{e1}
&s-s'=n_1-n_2+p_2-p_1, r-r'=\frac{n_2-n_1+k'-k+p_1-p_2}{2^{-1}}, i-i'=n_1-n_2+s-s'+r-r'+p_2-p_1.\qquad&
\ea
By substituting the first equation in \eqref{e1} into the third one, we have
\ba\label{e5}
i-i'=l-l'.
\ea
Since the vector fields \({\rm B}^{l}_{i,k}\) and \({\rm B}^{l'}_{i',k'}\) have the same \(\delta\)-grade,
\ba\label{e4}
i-i'=2(k'-k).
\ea Therefore, \(i-i'=l-l'=2(k'-k)\) and \(i+2k=i'+2k'.\) Hence, the inequalities \(i'\geq 0\) and \(-1\leq l'\leq 2i'+1\) are equivalent
to the inequalities \eqref{Restricted} on \(m:=k'\). The equality \(P^{l}_{i,k}=P^{m_1}_{m_2,m_3}\) is due to the definition.

Part \ref{c} follows equations \eqref{e4} and \eqref{e5}. The claim \ref{d} is a direct corollary of the claim \ref{b}. The proof of part
\ref{e} is similar to claims \ref{b}--\ref{d} and the claim is consistent with equation \eqref{Clebsch}.

Finally for the claim \ref{f}, consider
\(v:=\sum_{N=0}^{\infty}w_N\) for \(w_N\in \B_N.\) Then by claim \ref{e}, for any \(N\) we have
\bas
&w_{N}:=\sum_{j=-1}^{2N+1} w_{N}^j,\quad\, w_{N}^j\cdot {\bf e}_p\in \Span\, {\rm Terms}({\rm B}^j_{N,0}\cdot {\bf e}_p).&
\eas
Now we Taylor-expand \(w_N^j\) in terms of \(\Delta,\) that is,
\ba\label{VnoDelta}
w^j_{N}=\sum_{k=0}^{\min\{\lfloor\frac{2N-j+1}{4}\rfloor, \lfloor \frac{N}{2}\rfloor\}} v^{j}_{N-2k,0}\Delta^k,\quad \hbox{where}\quad
v^{j}_{N-2k,0}\cdot {\bf e}_p\in \Span\,
{\rm Terms}({\rm B}^{j}_{N-2k,0}\cdot {\bf e}_p).
\ea Hence, \(v^{j}_{N-2k,0}\) does not share any monomial vector field with any \({\rm B}\)-terms except \({\rm B}^{j}_{N-2k,0}.\)
This is because of the claim in part \ref{b} and that we have already excluded the powers of \(\Delta\) in \eqref{VnoDelta}. Therefore,
the expansion of \(v^{j}_{N-2k,0}\) in terms of the \({\rm B}\)-term generators of \(\B\) only includes \({\rm B}^{j}_{N-2k,0}.\)
\epr

\begin{rem}\label{commonac}
Given Lemma \ref{common}, the same statements trivially hold for other \(\Sl\)-generated vector fields from spaces \(\mathscr{A}\) and \(\mathscr{C}\) defined in
Appendix \ref{SecAppend}. In particular, for any \(v\in \mathscr{C}\) and \(w\in \mathscr{A},\) there exist uniquely determined constants \(c^{j}_{N, k}\) and
\(a^j_{N, k}\) so that
\ba\label{wCterm}
&v=\sum_{N=0}^{\infty}\sum^{2N}_{j=0}\sum_{k=0}^{\min\{\lfloor\frac{2N-j}{4}\rfloor, \lfloor \frac{N}{2}\rfloor\}}
c^{j}_{N, k}{\rm C}^{j}_{N-2k, 0} \Delta^k,\qquad&\\\label{wAterm}
&w=\sum_{N=0}^{\infty}\sum^{2N+2}_{j=-2}\sum_{k=0}^{\min\{\lfloor\frac{2N-j+2}{4}\rfloor, \lfloor \frac{N}{2}\rfloor\}} a^j_{N, k} {\rm A}^{j}_{N-2k, 0}\Delta^k.&
\ea
\end{rem}

The following theorem provides a concrete characterization for vector fields in \(\B\).
\begin{thm}\label{Bdivint}
Let \(v\) be a three dimensional vector field. The conditions \({\rm div}(v)=0\) and \(	v(\Delta)=0\) hold if and only if \(v\in \B.\)
\end{thm}
\bpr Given Theorem \ref{Solenoid} and Theorem \ref{int}(part \ref{FirstInts}), we only need to prove the {\it only if} part.
Let \(v=u+w_C+w_A,\) where \(u\in{\mathscr{B}},\) \(w_A\in \mathscr{A}\) and \(w_C\in \mathscr{C}.\) Hence,
\(\nabla\cdot (w_C+w_A)=0\) and \(w_C(\Delta)+w_A(\Delta).\)
By equations \eqref{wCterm} and \eqref{wAterm},
\bas
&w_C+w_A=&\sum_{N=0}^{\infty}\Big(\sum^{2N}_{j=0}\sum_{k=0}^{\min\{\lfloor\frac{2N-j}{4}\rfloor, \lfloor \frac{N}{2}\rfloor\}} {w_C}^{j}_{N-2k, 0}\Delta^k
+\sum^{2N+2}_{j=-2}\sum_{k=0}^{\min\{\lfloor\frac{2N-j+2}{4}\rfloor, \lfloor \frac{N}{2}\rfloor\}} {w_A}^{j}_{N-2k, 0}\Delta^k\Big),
\eas where \({w_C}^{j}_{N-2k, 0}\in\mathbb{R}\{{\rm C}^j_{N-2k,0}\}\) and \({w_A}^{j}_{N-2k, 0}\in \mathbb{R}\{{\rm A}^j_{N-2k,0}\}\). Since
\bes
{\rm Terms}((\mathbb{R}\{{\rm C}^{j_1}_{N-2k_1,0}\}+\mathbb{R}\{{\rm A}^{j_1}_{N-2k_1,0}\})\cdot {\bf e}_p)
\cap{\rm Terms}((\mathbb{R}\{{\rm C}^{j_2}_{N-2k_2,0}\}+\mathbb{R}\{{\rm A}^{j_2}_{N-2k_2,0}\})\cdot {\bf e}_p)= \varnothing
\ees for all \(p=1, 2, 3\) when \((j_1,k_1)\neq(j_2,k_2),\)
\bes{\rm div}\,{w_C}^{j}_{N-2k, 0}+ {\rm div}\,{w_A}^{j}_{N-2k, 0}=0\quad \hbox{ and }\quad {w_C}^{j}_{N-2k, 0}(\Delta)+{w_A}^{j}_{N-2k, 0}(\Delta)=0.\ees
By Theorem \ref{divA}, \({\rm div}\,{w_A}^{j}_{N-2k, 0}=0\) and so \({\rm div}\,{w_C}^{j}_{N-2k, 0}=0\) for all \(j, k, N.\)
By Lemma \ref{A.4}, equation \eqref{CEul}, and item \ref{f} in
Lemma \ref{common}, the condition \({\rm div}\,{w_C}^{j}_{N-2k, 0}=0\) implies that \({w_C}^{j}_{N-2k, 0}=0.\) Thereby, \({w_A}^{j}_{N-2k, 0}(\Delta)=0.\)
Now Theorem \ref{divA} concludes the proof through the equation \({w_A}^{j}_{N-2k, 0}=0\) for all \(j, k, N.\) Indeed, \(w_A= w_C=0\) and \(v=u\in \B.\)
\epr

\begin{thm} The following hold.
\begin{itemize}
\item[1.] The set \(\{ {\rm Terms}({\mathfrak {b}}^l_{N,0}):-1\leqslant l\leqslant 2N+1\}\) forms a partition for \(\mathfrak{B}_N.\)
\item[2.] For \(p=1, 2, 3,\) and \((x_1, x_2, x_3):= (x, y, z),\) \(\{x_p,{{\rm Terms}(\mathfrak{b}^l_{i,k})}\}\subseteq {{\rm Terms}(\mathfrak{b}^{l+2-p}_{i,k})}.\) When \((p, l)\neq (1, 2i+1),\)
\((p, l)\neq (2, i)\) and \((p, l)\neq (3, -1),\) \({\rm Terms}{\{x_p,{{\rm Terms}(\mathfrak{b}^l_{i,k})}\}}= {{\rm Terms}(\mathfrak{b}^{l+2-p}_{i,k})}.\)
\item[3.] \(\ker{\rm ad}_{x}= \mathbb{R}[[x]],\) \(\ker{\rm ad}_{y}= \mathbb{R}[[y, xz]],\) and \(\ker{\rm ad}_{z}= \mathbb{R}[[z]].\)
\end{itemize}
\end{thm}
\bpr
The first item follows items \ref{c} and \ref{d} in Theorem \ref{common}, and Lemma \ref{iso}. The second and third claim follows from
the structure constants and Lemma \ref{iso}.
\epr

An alternative representation for vector fields in \(\B\) are based on the vector potential. Each solenoidal vector field \(v\) has always a vector potential that is unique modulo
gradient vector fields. Vector potential frequently appears in the classical and quantum mechanics,
\eg see \cite{mackay1994transport}. Vector potential is called magnetic vector potential in electrodynamics while the curl of the magnetic vector potential
is called magnetic field; see \cite{barbarosie2011representation}.

\begin{thm}\label{vp0}
There exists a vector potential \(\phi^l_{i,k}\) such that \({\rm B}^l_{i,k}={\bf {curl}}(\phi^l_{i,k}),\)
where
\ba\label{vpf1}
&{\phi}^{l}_{i,k}:= {\mathfrak{b}^l_{i,k}}\nabla\Delta=\frac{{\Delta}^{k}\N^{l+1}\mb{{z}^{i+1}}}{{\kappa_{l+1, 2i+2}}}\big({z},-2{y},{x}\big).&
\ea
\end{thm}
\bpr
From equation \eqref{Clebsch} and the equality \(\nabla{f}\times \nabla {g}=\nabla \times f \nabla{g}\) for all scalar functions \(f\) and \(g,\) we have
\bas
{\rm B}^l_{i,k}= \nabla\times( {\mathfrak{b}^l_{i,k}}\nabla\Delta).
\eas
Thereby, \({\mathfrak{b}^l_{i,k}}\nabla\Delta\) is a vector potential for \({\rm B}^l_{i,k}.\)
\epr

\begin{rem}\label{vp}
An alternative vector potential for solenoidal vector fields is available through the computational approach on \cite[page 21]{mejlbroreal}.
Indeed, there exists a vector potential \(\Phi^l_{i,k}\) such that \({\rm B}^l_{i,k}={\bf {curl}}(\Phi^l_{i,k}),\) where
\ba\label{vpf}
&{\Phi}^{l}_{i,k}=\Big(\frac{(l-i){z}{\rm N}^{l+1}\mb{{z}^{i+1}}-(l+1){y}{\rm N}^{l}\mb{{z}^{i+1}}}{(i+1)(2k+i+3)\Delta^{-k}},
\frac{(2i+1-l){z}{\rm N}^{l+2}\mb{{z}^{i+1}}+(l+1){x}{\rm N}^{l}\mb{{z}^{i+1}}}{(i+1)(2k+i+3)\Delta^{-k}}, \frac{(l-2i-1)y{\rm N}^{l+2}\mb{{z}^{i+1}}
+(i-l)x{\rm N}^{l+2}\mb{{z}^{i+1}}}{(i+1)(2k+i+3)\Delta^{-k}}\Big).&
\ea In particular, \(\Phi^{1}_{0,0}=(-\Delta,0,0).\) The proof here follows \cite[page 21]{mejlbroreal}. Indeed, define
\bas
P(X):=\int_{0}^1t\,{{\rm B}^{l}_{i,k}}(tX)\,dt= \frac{1}{(2k+i+3)}{\rm B}^{l}_{i,k},
\eas
where \(X:= (x, y, z).\) Then, a vector potential for \({\rm B}^l_{i,k}\) is given by \(P(X)\times X.\) Hence, equation \eqref{vpf} is computed through
equation \eqref{Blik}.

We further recall that a vector potential for a given solenoidal vector field is generally unique modulo gradient vector fields. For instance
by equations \eqref{vpf1} and \eqref{vpf}, vector fields
\bas
&\phi^1_{1,0}=\left(\frac{{{y}}^{2}{z}}{4}, -\frac{{x}y{z}}{2}, \frac{{x}{{y}}^{2}}{4}\right) \; \hbox{ and } \;
\Phi^1_{1,0}=\frac{\left( {x}{z}+2{{y}}^{2}\right)\left(\frac{1}{2}{z}, -{y}, \frac{1}{2}{x}\right)}{3}&
\eas are two different vector potentials for \({\rm B}^1_{1,0}.\) Here, \(\phi^1_{1,0}+\nabla f=\Phi^1_{1,0}\) where
\(f=\frac{1}{12}{x}{{y}}^2{z}-\frac{1}{6} {{y}}^4+\frac{1}{12}{x}^2{{z}}^2.\)
\end{rem}

Nonlinear vector fields from \(\B\) are rotational vector fields; \ie they have a nonzero curl.

\begin{thm} \label{Curl}
All vector fields from \(\B\) have a non-zero curl. In particular, the null space of curl operator on the formal sum of vector field types \eqref{bm} is given by
\(\mathbb{R}\rm B^0_{0,0}.\)
\end{thm}
\bpr Note that \({\bf curl}({\rm B}^0_{0,0})=0.\) Let \(v\in \B\) and \({\bf curl} (v)=0.\) By equation \eqref{vterm}, Lemma \ref{common}, linearity and continuity
in filtration topology of \({\bf curl}\) operator,
\bes
v=\sum_{N=0}^{\infty}\sum^{2N+1}_{j=-1}\sum_{k=0}^{\min\{\lfloor\frac{2N-j+1}{4}\rfloor, \lfloor \frac{N}{2}\rfloor\}} v^{j}_{N-2k,0}\Delta^k,
\quad\hbox{ where} \quad v^{j}_{N-2k,0}\in{\mathbb{R}}\{{\rm B}^{j}_{N-2k,0}\},
\ees  and \({\bf curl} (v^{j}_{N-2k,0}\Delta^k)=0\) for all \(j, l, N.\) Now let \({\bf curl}({\rm B}^l_{i,k})=0.\) Hence, all three components of
\({\bf curl}({\rm B}^l_{i,k})\) are zero. The first component of \({\bf curl}({\rm B}^l_{i,k})\) is given by \(-\frac{\D}{\D{{y}}}
\frac{(l+1)\Delta^k{\rm N}^{l}\mb{{{z}}^{i+1}}}{(i+1)\kappa_{l,2i+2}}-\frac{\D}{\D{{z}}}\frac{(i-l)\Delta^k{\rm N}^{l+1}
\mb{{{z}}^{i+1}}}{(i+1)\kappa_{l+1,2i+2}}\) and by Lemma \ref{A.4}, we have
\ba\label{Curl0}
&\nabla\times{\rm B}^{l}_{i,k}\cdot{{\bf e}_1}=-\frac{2l (l+1)\Delta^k{\rm N}^{l-1}\mb{{{z}}^{i}}}{\kappa_{l,2i+2}}+\frac{2{y}k(l+1)
\Delta^{k-1}{\rm N}^{l}\mb{{{z}}^{i+1}}}{(i+1)\kappa_{l,2i+2}}
-\frac{(i-l)\Delta^k{\rm N}^{l+1}\mb{{{z}}^{i}}}{\kappa_{l+1,2i+2}}-\frac{k(i-l){x}\Delta^{k-1}{\rm N}^{l+1}\mb{{{z}}^{i+1}}}{(i+1)\kappa_{l+1,2i+2}}=0.&
\ea Since \({\rm Terms}(\Delta^k{\rm N}^{l-1}\mb{{{z}}^{i}}),\) \({\rm Terms}(\Delta^{k-1}{\rm N}^{l}\mb{{{z}}^{i+1}}),\)
\({\rm Terms}(\Delta^k{\rm N}^{l+1}\mb{{{z}}^{i}}),\) and \({\rm Terms}({\rm N}^{l+1}\mb{{{z}}^{i+1}})\) are pairwise disjoint sets of monomial terms,
equation \eqref{Curl0} holds if and only if
\(l(l+1)=k(l+1)=(i-l)=k(i-l)=0.\) The later is equivalent with \(i=k=l=0.\) This completes the proof.
\epr


\begin{exm}
Let
\bas
v:={\rm A}^{0}_{0,1}-3{\rm C}^{2}_{2,0}=-3{x}{{y}}^{2}\frac{\D}{\D {x}}-3{x}{y}
{z}\frac{\D}{\D {y}}-3{{y}}^{2}{z}\frac{\D}{\D {z}}.
\eas Then, \(v(\Delta)=0\) while \({\rm div}(v)=-3{x}{z}-6{{y}}^2\neq0.\)

Consider the vector field
\bas
{\rm A}^{-2}_{i,0}:=z^{i+1}\frac{\D}{\D{{x}}}, \quad \hbox{ for any } i\in \NZ.
\eas
This family has two first integrals of \(y\) and \(z\) while \({\rm A}^{-2}_{i,0}\) is also solenoidal for all \(i\). These vector fields do
not generate a Lie algebra with the nilpotent linear part \({\rm B}^1_{0,0},\) indeed,
\bas
&[{\rm A}^{-2}_{j,0}, {\rm A}^{-2}_{i,0}]=0, \qquad {[{\rm B}^1_{0,0}, {\rm A}^{-2}_{i,0}]}=-2(i+1) {\rm A}^{-1}_{i,0}.&
\eas This indicates that the family of vector fields in \(\B\) does not represent the set of all solenoidal vector fields with two independent first integrals.
\end{exm}

\section{Normal form classification } \label{Sec5}


\begin{thm}\label{UNF} The vector field \eqref{eq1}-\eqref{eq2} is either linearizable in the first level normalization step or there exists a
natural number \(p\) so that the first level normal form of the vector field \eqref{eq1}-\eqref{eq2} is given by
\ba\label{SNF}
&{\bf w}:=-{x}\frac{\D}{\D{{y}}}-2{y}\frac{\D}{\D{{z}}}+{{z}}^{p} ({z}\frac{\D}{\D{{y}}}
+2{y}\frac{\D}{\D{{x}}})+\sum_{k=p+1}^{\infty}\sum_{i=0}^{[{\frac{k+p}{2}}]} b_{i,k}{{z}}^i(xz-y^2)^{k-2i+p}({z}\frac{\D}{\D{{y}}}
+2{y}\frac{\D}{\D{{x}}}),&
\ea where \({b}_{i,k}\in \mathbb{R}\). Furthermore, the normal form vector field \eqref{SNF} is always an the infinite level normal form.
In addition, the infinite level secondary Clebsch potential normal form is given by
\ba\label{invSNF}
&I(x, y, z)={x}+\frac{1}{p+1}{{z}}^{{p}+1}+\sum_{k=p+1}^{\infty}\sum_{i=0}^{[{\frac{k+p}{2}}]}
{\frac{b_{{{i},k}}}{{i}+1}}{{z}}^{{i}+1}{( {x}{z}-{{y}}^{2} )}^{k-2i+p}.&
\ea Here, the quadratic polynomial \(\Delta=xz-y^2\) stands for the primary Clebsch potential.
\end{thm}
\bpr
The normal form
\ba\label{nf}
&{\bf w}:={\rm B}^1_{0,0}+\sum_{i,k\in{\mathbb{N}_0}}^{\infty} {b}_{i,k} {\rm B}^{-1}_{i,k},&
\ea is readily available given the \(\Sl\)-style normal form and
the fact that \(\ker({\rm ad}_{{\rm M}})=\Span\{ {\rm B}^{-1}_{i,k}\}\); for more information see \cite{baidersanders,BaidSand91}.
Let \({r}:=\min\{i\mid{b}_{i,0}\neq0\}\) and \(r>0.\) Define a new grading function by \(\delta({\rm B}^{l}_{i,k}):={l{r}+2i+k}.\) Then, \(\B\) is a \(\delta\)-graded
Lie algebra and
\bas
\mathbb{B}_{p} := {\rm B}_{0,0}^1 +{b}_{p,0}{\rm B}_{{p},0}^{-1}\in \B_p.
\eas Linear invertible transformations can be used to rescale the coefficient \(b_{p,0}\) into \(b_{p,0}:=1.\)
Following \cite{baidersanders,BaidSand91} we define
\ba\label{gamma}
\Gamma:= \ad ({\rm B}^{-1}_{0,0})\circ \ad(\mathbb{B}_{p}).
\ea By the structure constants and equation \eqref{S3},
\bas
\Gamma ({\rm B}^{q}_{i,k})= (q-2i-1)(q+2) {\rm B}^{q}_{i,k}+b_{p,0}\sum_{j=\max\{q-2-i-p, -1\}}^{s-1+{\lfloor\frac{r+1}{2}\rfloor}}\frac{{{a}}_{j,p,i}^{-1,q}
{\rm B}^{2j+|r-1|-1}_{2j-q+1+i+p+|r-1|,s-1+k-j+{\lfloor\frac{r+1}{2}\rfloor}}}{(2j+|r-1|+1)^{-1}}.
\eas Hence for any \(i\) and \(k\) (\(q=2i+1\)), there is a possibility of a vector polynomial in kernel \(\Gamma.\) This is due to a similar argument
used by \cite{baidersanders,BaidSand91}. On the other hand
\(\mathfrak{b}^{2i+1}_{i,k}=-{x}^{i+1}\Delta^k,\) \(\Psi((\mathfrak{b}^1_{0,0})^{i+1}\Delta^k)={\rm B}^{2i+1}_{i,k},\) and
\bas
\mathbb{B}_{{r}}^{i+1}\Delta^k:=\Psi\left((\Psi^{-1}(\mathbb{B}_{{r}}))^{i+1}\Delta^k\right)=\Psi((-{x}-{z}^{{r}+1})^{i+1}\Delta^k)\in \ker \Gamma.
\eas These polynomial vectors are extended to a symmetry for the normal form vector field \eqref{SNF}, through
\bas
\Psi\left((\Psi^{-1}({\bf w})^{i+1}\Delta^k\right).
\eas This proves that there is no possibility of any hypernormalization beyond the \(\Sl\)-style normalized vector field \eqref{SNF}.
\epr
\begin{cor}
The following presents five alternative representations for the  normal form \eqref{SNF}:
\begin{enumerate}
\item A formal sum of \({\rm B}\)-terms:
\bas
&{\bf w}:={\rm B}^1_{0,0}+ {\rm B}^{-1}_{p,0}+\sum_{k=p+1}^{\infty}\sum_{i=0}^{[{\frac{k+p}{2}}]}  {b}_{i,k} {\rm B}^{-1}_{i,k+p-2i}.&
\eas
\item The secondary invariant:
\bas
&{\bf w}:=\Psi\big(-{x}-{{z}}^{p}- \sum_{k=p+1}^{\infty}\sum_{i=0}^{[{\frac{k+p}{2}}]} {b_{i,k}{z}^{i+1}({x}{z}-{{y}}^2)^{k+p-2i}}\big).&
\eas Here, \(\Psi\) is the Lie isomorphism given by equation \eqref{fiso}.
\item Vector potential:
\bas
&{\bf w}:={\bf curl}\Big(({x}+ {{z}}^{p+1}+\sum_{k=p+1}^{\infty}\sum_{i=0}^{[{\frac{k+p}{2}}]}\frac{\Delta^k {{z}}^{i+1}}{i+1})[{z}, -2{y}, {x}]\Big).&
\eas
\item Functionally independent Clebsch potentials:
\bas
&{\bf w }:=\nabla({x}{z}-{{y}}^2)\times\nabla\Big({x}+ {{z}}^{p+1}+ \sum_{k=p+1}^{\infty}
\sum_{i=0}^{[{\frac{k+p}{2}}]} \frac{b_{{{i},k}}{{z}}^{{i}+1}{( {x}{z}-{{y}}^{2} )}^{k+p-2i}}{{i}+1}\Big).&
\eas
\item Poisson bracket:
\bas
&{\bf w }=\sum_{p=1}^3\{x_p, {\bf I}(x, y, z)\}\cdot{{\bf e}_p}&
\eas where \({\bf I}(x, y, z)\) is the invariant given in equation \eqref{invSNF}.
Furthermore, \(\{{\bf I}(x, y, z),\Delta\}=0.\)
\end{enumerate}
\end{cor}
\bpr Follow equation \eqref{Blik}, Lemma \eqref{iso}, item \(2\) in Theorem \eqref{int}, and Theorem \ref{thmBpoisson}, respectively.
\epr

Now we consider further reduction of a normal form system given by
\ba\label{SNF2}
&{\bf w}:=-{x}\frac{\D}{\D{{y}}}-2{y}\frac{\D}{\D{{z}}}+{{z}}^{p} ({z}\frac{\D}{\D{{y}}}
+2{y}\frac{\D}{\D{{x}}})+\sum_{i=p+1}^{\infty} b_{i}{{z}}^i({z}\frac{\D}{\D{{y}}}
+2{y}\frac{\D}{\D{{x}}}).&
\ea

\begin{thm}\label{Hamiltonian}
There exists a near identity transformation so that the vector field \eqref{SNF2} is transformed into the normal form vector field
\ba\label{bif2}
&v:=-2Y\frac{\D}{\D{Z}}+\left(-X+\frac{p+2}{p+1}Z^{p+1}+\sum_{i=p+1}^{\infty}\frac{b_{i}(i+2)}{i+1}Z^{i+1}\right)
\frac{\D}{\D{Y}}.&
\ea Furthermore, \(X(t):=c\) is always constant. Hence, the normal form system \eqref{bif2} has a Hamiltonian
\ba\label{ham}
&H(Z, Y):= -\Delta(X, Y, Z) =Y^2-cZ+\frac{1}{p+1}Z^{p+2}+\sum_{i=p+1}^{\infty}\frac{b_{i}}{i+1}Z^{i+2}.&
\ea On the invariant manifold \(I(x, y, z)=0,\) the normal form vector field takes a further hypernormalization given by
\bas
&v^{(\infty)}=-2 {y} \frac{\D}{\D{{z}}}+\sum_{i=p}^{\infty}a_i{{z}}^{{i}+1}\frac{\D}{\D{{y}}},&
\eas where \(a_i=0\) for \( i=(p+1)(m+1),\) \(m\in \mathbb{N}_0.\)
\end{thm}
\bpr The key idea is to use the secondary Clebsch potential \eqref{invSNF} as a near-identity transformation, \ie
\ba\label{trans}
&(X, Y, Z):= I(x, y, z) =\left(x+\frac {1}{p+1}{{z}}^{{p}+1}+\sum_{i=p+1}^{\infty}
\frac{b_{i}}{i+1}z^{i+1}, y, z\right).&
\ea Hence, \(X(t)\) is constant. Then, the normal form vector field is given by
\bas
{\bf h}:=2A^{1}_0-cA^{-1}_{-1}+ \frac {(p+2)}{p+1}A^{-1}_{p}+\sum_{i=p+1}^{\infty}\frac {(i+2)b_{{{i}}}}{i+1} A^{-1}_{i},
\eas in terms of notations used in \cite{BaidSand91,GazorMoazeni}. Hence, the second claim follows \cite[Theorem 8.9]{BaidSand91}.
\epr

\subsection{Truncated normal form coefficients }\label{Symbolic}

Consider a cubic-degree truncated triple zero vector field
\ba\label{sys3}
&v:= -{x}\frac{\D}{\D{{y}}}-2{y}\frac{\D}{\D{{z}}}+ \sum^3_{i+j+k=2} {{x}}^i {{y}}^j {{z}}^k\big(  a_{ijk}\frac{\D}{\D{{x}}}+ b_{ijk} \frac{\D}{\D{{y}}}
+ c_{ijk}\frac{\D}{\D{{z}}}\big)\in \B,&
\ea where \(a_{ijk},b_{ijk}\) and \(c_{ijk}\in{\mathbb{R}}.\) By Theorem \ref{Bdivint}, \({\rm div} (v)=0\) and \(v(\Delta)=0.\) The first one implies
\ba\nonumber
&c_{002}= -\frac{1}{2}(a_{101}+b_{011}),\quad a_{200}=-b_{110},  \quad c_{020}= 2b_{110}, \quad
c_{101}= -(2 a_{200}+b_{110}),  \quad a_{101}=b_{011},&
\\\label{C2}
&a_{020}=2b_{011}, \quad c_{011}= -a_{110},  \quad\quad a_{011}=2b_{002},  \qquad\quad  c_{110}=2b_{200}, \quad c_{200}= a_{002}=b_{101}=0,&
\ea while \(v(\Delta)=0\) concludes \(c_{300} =a_{003} = 0\) and
\ba\label{C3}
\nonumber
&c_{111}= -2(a_{210}+b_{120}), \;  b_{012} = -2c_{003}, \; b_{111}=-2(a_{201}+c_{102}), \; a_{012} = 2b_{003}, \;  c_{210} =2b_{300}, \;a_{120}=-c_{021},&
\\\nonumber
&a_{111} = -2b_{021}-2c_{012}, \; b_{210} = -2a_{300}, \; a_{102}=-b_{012}-3c_{003}, \;c_{201}=-3a_{300}-b_{210}, \; a_{030} = 2b_{021}, &
\\\nonumber
&c_{030} =2b_{120}, c_{012}=-2(b_{021}+b_{102}), a_{201} = -c_{102},  a_{210}=-2(b_{201}+b_{120}),  a_{021} = -4c_{003},  c_{120}=-4a_{300}.&
\ea Hence, the vector field \eqref{sys3} can be written as
\ba\label{sys03}
&v={\rm B}^1_{{0,0}}+{d}^{-1}_{1,0}{\rm B}^{-1}_{{1,0}}+{d}^{0}_{1,0}{\rm B}^0_{{1,0}}+
{d}^{1}_{1,0}
{\rm B}^1_{{1,0}}+{d}^{2}_{1,0}{\rm B}^2_{{1,0}}+{d}^{3}_{1,0}{\rm B}^3_{{1,0}}+
{d}^{-1}_{0,1}{\rm B}^{-1}_{{0,1}}+{d}^{-1}_{2,0}{\rm B}^{-1}_{{2,0}}&
\\\nonumber
&\qquad\qquad\quad
+{d}^{1}_{0,1}{\rm B}^1_{{0,1}}+{d}^{0}_{0,1}{\rm B}^0_{{0,1}}+{d}^{1}_{2,0}{\rm B}^1_{{2,0}}+{d}^{2}_{2,0}{\rm B}^2_{{2,0}}
+{d}^{3}_{2,0}{\rm B}^3_{{2,0}}+{d}^{4}_{2,0}{\rm B}^4_{{2,0}}+{d}^{5}_{2,0}{\rm B}^5_{{2,0}}+{d}^{0}_{2,0}{\rm B}^0_{{2,0}},&
\ea where \({d}^{-1}_{1,0}=b_{002},\; {d}^{-1}_{0,1}=\frac{1}{5}(4b_{102}-b_{021}), \; {d}^{1}_{2,0}=(3b_{021}+3b_{102}),\) and
\bas
&{d}^{0}_{1,0}=2 b_{011}, \; {d}^{0}_{2,0}=-3c_{003}, \; {d}^{2}_{2,0}=-(c_{021}+c_{102}), \; {d}^{1}_{1,0}= a_{110}, \; {d}^{5}_{2,0}=-b_{300},\;
{d}^{-1}_{2,0}=b_{003}, \; {d}^{2}_{1,0}=-2 b_{110}, &
\\
&{d}^{4}_{2,0}=3a_{300},\; {d}^{1}_{0,1}=\frac{1}{5}(b_{120}-4 b_{201}), \; {d}^{3}_{1,0}=-b_{200},\; {d}^{3}_{2,0}=-(3b_{201}+3b_{120}),
\; {d}^{0}_{0,1}=\frac{1}{5}(c_{021}-4 c_{102}).&
\eas

\begin{prop}
The quartic truncated normal form for equation \eqref{sys03} is given by
\bas
&w= -{x}\frac{\D}{\D{{y}}}-2{y}\frac{\D}{\D{{z}}}+\Big(2{y}\big(
({x}{{z}}-{{y}}^{2})({b}^1_{1}{{z}}+{b}^0_{1})+
{z}({b}^3_{0}{{z}}^{2}+ {b}^2_{0}{{z}}+{b}^1_{0}
)\big)\Big)\frac{\D}{\D{{x}}}&
\\
&+\big({z}({x}{{z}}-{{y}}^{2})({b}^1_{1}{{z}}+{b}^0_{1} )+{{z}}^{2}({b}^3_{0}{{z}}^{2}+ {b}^2_{0}{{z}}+{b}^1_{0})\big)\frac{\D}{\D{{y}}},&
\eas
whose first integrals are \(\Delta ={x}{z}-{{y}}^2\) and
\bas
&I(x, y, z)={x}-{b}^0_{1} {z}{{y}}^{2}+\frac{{z}^2}{2}({b}^1_{1}{{x}}{z}-{b}^1_{1}{{y}}^{2}+4{b}^1_{0}+2{b}^0_{1} {{z}}^{2}
+\frac{{b}^3_{0}{{z}}^{2}}{2}+\frac{2{b}^2_{0}{{z}}}{3}),&
\eas where
\bas
&{b}^1_{0}=b_{002}, \qquad\quad { b}^2_{0}=b_{003}+\frac{b_{002}a_{110}}{2}-\frac{3{b_{011}}^{2}}{4}, \qquad\quad
{b}^0_{1}= {\frac {{a_{110}}^{2}}{60}}+\frac{b_{110}b_{011}}{5}-\frac{b_{200}b_{002}}{5}+\frac{1}{5}(4b_{102}-b_{021}),&
\\&{b}^1_{1}= -{\frac {{a_{110}}^{3}}{378}}-\frac{{b_{110}}^{2}b_{002}}{7}+\frac{(4 c_{102}-c_{021})b_{011}}{15}
+\frac{{b_{011}}^{2}b_{200}}{21}+{\frac {12(b_{201}+b_{120})b_{002}}{105}}+\frac{8b_{110}c_{003}}{21}
+{\frac {12(b_{021}+b_{102})a_{110}}{105}}-\frac{4b_{200}b_{003}}{7}&
\\&
+\frac{2b_{011}(c_{021}+c_{102})}{35}-\frac{2b_{200}b_{002}a_{110}}{21}-{\frac {2b_{110}b_{011}a_{110}}{63}}
+\frac{2}{5}(b_{120}-4 b_{201})b_{002}+\frac{(4b_{102}-b_{021})a_{110}}{15},\qquad\qquad&
\\
&{b}^3_{0}=\frac{2 b_{003}a_{110}}{3}-\frac{6b_{110}b_{002}b_{011}}{5}+{\frac {4(3b_{021}+3b_{102})b_{002}}{15}}
-\frac{6b_{200}{b_{002}}^{2}}{5}+\frac{b_{002}{a_{110}}^{2}}{10}+{2c_{003}b_{011}}.\qquad\qquad\qquad\qquad\qquad&
\eas The vector potential normal form for vector field \eqref{sys03} is given by
\bas
&\big({x}+b^1_0\frac{ {{z}}^{2}}{2}+b^0_1 {\Delta {{z}}}+b^2_0\frac{ {{z}}^{3}}{3}+b^1_1\frac{\Delta {{z}}^{2}}{2}
+b^3_0\frac{ {{z}}^{4}}{4}\big)(z, -2y, x).&
\eas
\end{prop}
\bpr The normal form coefficient are derived using an implementation of the formulas in Maple.
\epr

\appendix

\section{Appendix}\label{SecAppend}

Let
\begin{eqnarray}\label{Av}
&\mathscr{A}:= \Span \left\{
\sum a^l_{i,k}{\rm A}^{l}_{i,k}\big|  -2\leqslant l \leqslant 2i+2, i, k\in\mathbb{N}_0, a^l_{i,k}\in{\mathbb{R}}\right\},&
\end{eqnarray}
and
\ba\label{CFamily}
&\mathscr{C}:= \Span \left\{\sum c^l_{i,k}{\rm C}^{l}_{i,k}\big|\,  0\leqslant l \leqslant 2i, i, k\in\mathbb{N}_0, c^l_{i,k}\in{\mathbb{R}}\right\},&
\ea where
\ba\label{am}
{\rm A}^{l}_{i,k}:= \frac{ \N^{l+2}{{z}}^{i}\Delta^k\mb{\frac{\D}{{\D x}}}}{{\kappa_{l+2, 2i+2}}}, \qquad \hbox{for }\quad
-2\leq l\leq 2i+2, i, k\in \mathbb{N}_0:=\mathbb{N}\cup\{0\},
\\\label{cm}
{\rm C}^{l}_{i,k}:=\frac{ \N^{l}z^i \Delta^k\mb{\rm E}}{\kappa_{l, 2i}}, \qquad\qquad\qquad \hbox{for }\quad 0\leq l\leq 2i, i, k\in \mathbb{N}_0:=\mathbb{N}\cup\{0\},
\ea and \({\rm E}= x\frac{\D}{{\D x}}+y\frac{\D}{\D y}+z\frac{\D} {{\D z}}.\) By equation \eqref{NfM}, we have
\ba\nonumber
&{\rm A}^l_{i,k}:=\frac{ \Delta^k\N^{{l+2}}\mb{z^{i+1}}}{\kappa_{l+2, 2i+2}}{\frac{\D}{\D x}}
- \frac{(l+2)\Delta^k\N^{l+1}\mb{z^{i+1}}}{(2i-l+1)\kappa_{l+1, 2i+2}}{\frac{\D}{\D y}}
+\frac{(l+1)(l+2)\Delta^k\N^{l}\mb{z^{i+1}}}{(2i-l+1)(2i-l+2)\kappa_{l, 2i+2}}\frac{\D}{\D z},&
\\\label{CEul}&{\rm C}^{l}_{i,k}:=\frac	{x\Delta^k\N^{l}\mb{z^{i}}}{\kappa_{l, 2i}}\frac{\D}{{\D x}}
+\frac{y\Delta^k\N^{l}\mb{z^{i}} }{\kappa_{l, 2i}}\frac{\D}{\D y}+
\frac{z\Delta^k\N^{l}\mb{z^{i}}}{\kappa_{l, 2i}}\frac{\D} {{\D z}}.\qquad\qquad\qquad\qquad\quad\;&
\ea

\begin{thm}\label{divA} For all	\(-2\leqslant l \leqslant 2i+2, i, k\in\mathbb{N}_0,\)
\be {\rm div}({\rm A}^l_{i,0})=0, \qquad
\nabla{\Delta}\cdot {\rm A}^l_{i, k}\neq0.
\ee
\end{thm}
\bpr
By equation \eqref{deriv} and definition \({\rm A}^l_{i,k},\) we have
\bas
&{\nabla} \cdot{\rm A}^l_{i,0}=(l+1)(l+2)(i+1)\big(\frac{1}{\kappa_{l+2,2i+2}}
-\frac{2}{\kappa_{l+1,2i+2}(2i-l+1)}
-\frac{1}{\kappa_{l,2i+2}(2i-l+1)(2i-l+2)}\big)\N^{ l}\mb{{z}^{i}}.&
\eas
Since  the coefficient \(\N^{ l}\mb{{z}^{i}}\) is zero, \({\rm div}({\rm A}^l_{i,k})=0\)
for any \(i\in\mathbb{N}_0,\) and \(-2\leqslant l \leqslant 2i+2 .\)

When \(l\) is even, say \(l=2s,\) by Lemma \ref{A.1} we have
\bas
&\nabla{\Delta}\cdot {\rm A}^l_{i,0}=
\frac{(2s+2)!(i+1)!(2i-2s)!}{(2i+2)!}\sum _{n=0}^{s+1}{\frac{{4}^{n}{x}^{s-n+1}{y}^{2n}{z}^{i-n-s+1}}{(s-n+1)!(2n)!(i-n-s)!}}
+\frac{(2s+2)!(i+1)!(2i-2s)!}{{(2i+2)!}}\sum_{n=0}^{s}{\frac{{4}^{n}{x}^{s-n+1}{y}^{2n}{z}^{i-n-s+1}}{(s-n)!(2n)!(i-n-s+1)!}}&
\\&+\frac{4(2i-2s)!(s+1)}{{(2i+2)!}}\sum_{n=1}^{s+1}{\frac{(2s+1)!(i+1)!{2}^{2n-1}{x}^{s-n+1}{y}^{2n}{z}^{i-n-s+1}}{(s-n+1)!(2n-1)!(i-n-s+1)!}}.\qquad\qquad\qquad\qquad&
\eas
Since  the coefficient of \({x}^{s+1}{z}^{i-s+1}\) is
\bas
{\frac{\left(i+2\right)!\left( 2i-2s\right)!\left(2s+2\right)!}{\left(2i+2\right)!\left(s+1\right)!\left(i-s+1\right)!}}\neq0,
\eas \(\Delta\) is not a first integral for \({\rm A}^{l}_{i,k}.\)
The argument is similar for when  \(l\) is odd.
\epr
\begin{thm}
Each three dimensional vector field \(v\) can be uniquely expanded in terms of formal sums of polynomial generators \({\rm A}^{l}_{i,k},\) \({\rm B}^{l}_{i,k}\) and
\({\rm C}^{l}_{i,k}\) from \(\mathscr{A},\) \(\B\) and \(\mathscr{C},\) respectively.
\end{thm}

\bpr
This is a straightforward corollary of Theorem \ref{3DVects}, Lemma \ref{common} and Remark \ref{commonac}.
\epr
\bibliographystyle{plain}
\bibliography{BTZ}
\end{document}